\documentclass[a4paper,reqno]{amsart}
\usepackage{amsmath}
\usepackage{amsthm}
\usepackage{amssymb}
\usepackage{mathtools}
\usepackage{multirow}
\usepackage{fullpage}
\usepackage{enumitem}
\usepackage{multicol}
\usepackage{longtable}
\usepackage{csquotes}
\newcommand{\II}{I\! I}
\newcommand{\legendre}[2]{\left(\frac{#1}{#2}\right)}
\newcommand{\Z}{\mathbb{Z}}
\newcommand{\Q}{\mathbb{Q}}
\newcommand{\C}{\mathbb{C}}
\newcommand{\R}{\mathbb{R}}
\renewcommand{\P}{\mathbb{P}}

\renewcommand{\H}{\mathbb{H}}
\newcommand{\e}{\mathfrak{e}}
\renewcommand{\Im}{\operatorname{Im}}

\DeclareMathOperator{\sign}{sign}

\DeclareMathOperator{\Aut}{Aut}
\DeclareMathOperator{\SL}{SL}

\DeclareMathOperator{\oddity}{oddity}
\DeclareMathOperator{\Id}{Id}
\DeclareMathOperator{\Co}{Co}

\newtheorem{lemma}{Lemma}[section]
\newtheorem{proposition}[lemma]{Proposition}

\newtheorem{theorem}[lemma]{Theorem}
\newtheorem*{theorem*}{Theorem}

\numberwithin{equation}{section}

\begin{document}
\title{Reflective automorphic forms on lattices of squarefree level}
\author{Moritz Dittmann}
\date{}
\begin{abstract}
We show that there are only finitely many nonconstant reflective automorphic forms $\Psi$ on even lattices of squarefree level splitting two hyperbolic planes and give a complete classification in the case where the zeros of $\Psi$ are simple and $\Psi$ has singular weight.
\end{abstract}
\maketitle
\section{Introduction}
Let $L$ be an even lattice of signature $(n,2)$ with dual lattice $L'$. The Hermitian symmetric domain $\mathcal{H}_L$ associated to $L$ is one of the two connected components of 
\[
\{[Z]\in \P(L\otimes\C) : (Z,Z)=0, (Z,\overline{Z})<0\}.
\]
Let $O(L)^+$ be the subgroup of $O(L)$ that preserves $\mathcal{H}_L$.
Automorphic forms are homogeneous functions on the affine cone over $\mathcal{H}_L$ that are invariant under a finite index subgroup of the group $O(L)^+$ of isometries fixing $\mathcal{H}_L$. One way to obtain such functions is to apply Borcherds' singular theta correspondence (see \cite{BorcherdsGrass}) to a suitable vector-valued modular form $F$. The resulting function $\Psi$ is an automorphic form for a group containing the kernel $\Gamma(L)$ of $O(L)^+\to \Aut(L'/L)$ and has product expansions at the cusps which is why $\Psi$ is also called an automorphic product. The divisor of $\Psi$ is a linear combination of rational quadratic divisors with multiplicities determined by the principal part of $F$. Conversely, Bruinier has shown in \cite{BruinierConverse} that if $L$ splits two hyperbolic planes, then every automorphic form $\Psi$ for $\Gamma(L)$ whose divisor is a linear combination of rational quadratic divisors is an automorphic product.

An automorphic form $\Psi$ is called reflective if its zeros are orthogonal to roots of $L$ and strongly reflective if all these zeros are simple. Such functions play an important part in finding interesting generalized Kac-Moody algebras, as their denominator identities should be reflective automorphic forms (see e.g. \cite{GritsenkoKacMoody}), and in algebraic geometry, as the existence of a reflective modular form of large weight (when compared to the vanishing orders of its zeros) implies that the corresponding modular variety is uniruled (see \cite{GritsenkoUniruled}). For $n\geq 4$ Ma has proven that up to scaling only finitely many lattices  carry a reflective automorphic form with bounded slope (see \cite{Ma}) and that there are only finitely many lattices carrying a $2$-reflective automorphic form (i.e. a reflective automorphic form whose zeros correspond to roots of norm $2$) (see \cite{Ma2}) and Scheithauer obtained partial classification results for strongly reflective automorphic products of singular weight on lattices of squarefree level (see \cite{ScheithauerClass} and \cite{ScheithauerSingular}). 

In the present paper we restrict to the case where the even lattice $L$ has squarefree level $N$ and additionaly assume that $L$ splits $\II_{1,1}\oplus \II_{1,1}(N)$. For such lattices we prove the following result.
\begin{theorem}\label{thm1}
There are only finitely many even lattices $L$ of signature $(n,2), n\geq 4$ and squarefree level $N$ that split $\II_{1,1}\oplus \II_{1,1}(N)$ and carry a nonconstant reflective automorphic form.
\end{theorem} 
Our result is effective, i.e. we obtain bounds on the level $N$ and the rank $n+2$ of $L$, thus obtaining a finite list of possible candidates for $L$. The idea of the proof is as follows. Suppose $\Psi$ is a nonconstant reflective automorphic form on $L$. By Bruinier's converse theorem above we can assume that $\Psi$ is an automorphic product corresponding to some vector-valued modular form $F$.  Since $L$ splits $\II_{1,1}$, the reflectivity of $\Psi$ can be restated as a condition on the principal part of $F$. For squarefree level $N$ every symmetric weakly holomorphic modular form for the Weil representation and in particular the symmetrization of $F$ is the lift of some scalar-valued weakly holomorphic modular form $f$ for the group $\Gamma_0(N)$. Then the condition on the principal part of $F$ can be restated as conditions on the principal parts of $f$ at all cusps (this uses that $L$ splits $\II_{1,1}(N)$). These conditions are so restrictive that there are only finitely many solutions for the level $N$ and the rank $n+2$.

One way to construct reflective automorphic forms is described in \cite{ScheithauerClass}: The automorphism group of the Leech lattice $\Lambda$ is Conway's group $\Co_0$. Let $g\in \Aut(\Lambda)$ have order $n$. Then the characteristic polynomial of $g$ can be written as $\prod_{k\mid n}(x^k-1)^{b_k}$. We associate to $g$ the eta product
\[
\eta_g(\tau)=\prod_{k\mid n}\eta(k\tau)^{b_k}
\] 
and assume that its level is squarefree. Moreover we assume that the fixpoint lattice $\Lambda^g$ of $g$ is nontrivial. The function $f_g=1/\eta_g$ can be lifted to a vector-valued modular form $F_g=F_{\Gamma_0(N),f_g,0}$ for the Weil representation on the lattice $L=\Lambda^g\oplus \II_{1,1}\oplus \II_{1,1}(N)$. The function $F_g$ can then be lifted to an automorphic product $\Psi_g$ with Borcherds' lift. This automorphic product turns out to be strongly reflective and to have singular weight. Further possible inputs can be obtained by applying Aktin-Lehner involutions to $f_g$ and $L$. With this method one obtains strongly reflective automorphic products of singular weight on the following lattices by lifting the function in the second column first to a vector-valued modular form $F_{\Gamma_0(N),f,0}$ and then to an automorphic product.

\renewcommand{\arraystretch}{1.2}
\begin{table}[h!]
\caption{}
\label{tabIntro}
\begin{tabular}{ccccc}
$L$ & $f$ & \hspace{2cm} & $L$ & $f$ \\ \cline{1-2} \cline{4-5}
$\II_{4,2}(2_{\II}^{+2}3^{+3})$ & $\eta_{1^{-1}2^23^36^{-6}}$ & & $\II_{6,2}(2_{\II}^{-6}5^{-5})$ & $\eta_{1^{2}2^{-3}5^{-2}10^{-1}}$ \\
$\II_{4,2}(2_{\II}^{+2}3^{+3})$ & $-3\eta_{1^32^{-6}3^{-1}6^{2}}$ & &  $\II_{6,2}(2_{\II}^{+2}5^{+5})$ & $\eta_{1^{-3}2^{2}5^{-1}10^{-2}}$  \\
$\II_{4,2}(2_{\II}^{-4}3^{-3})$ & $-2\eta_{1^22^{-1}3^{-6}6^{3}}$ & & $\II_{6,2}(2_{\II}^{+2}5^{+3})$ & $\eta_{1^{-1}2^{-2}5^{-3}10^{2}}$  \\
$\II_{4,2}(2_{\II}^{-4}3^{-3})$ & $6\eta_{1^{-6}2^{3}3^{2}6^{-1}}$ & & $\II_{6,2}(11^{-4})$ & $\eta_{1^{-2}11^{-2}}$ \\
$\II_{4,2}(2_{\II}^{+4}7^{-3})$ & $\eta_{1^{1}2^{-2}7^{1}14^{-2}}$ & & $\II_{6,2}(2_{\II}^{+4}7^{-4})$ & $\eta_{1^{-1}2^{-1}7^{-1}14^{-1}}$ \\
$\II_{4,2}(2_{\II}^{+2}7^{-3})$ & $2\eta_{1^{-2}2^{1}7^{-2}14^{1}}$ & & $\II_{6,2}(3^{+4}5^{-4})$ & $\eta_{1^{-1}3^{-1}5^{-1}15^{-1}}$ \\ \cline{4-5}
$\II_{4,2}(3^{+3}5^{-3})$ & $\eta_{1^{-2}3^{1}5^{1}15^{-2}}$ & & $\II_{8,2}(3^{-7})$ & $\eta_{1^{3}3^{-9}}$ \\
$\II_{4,2}(3^{-3}5^{+3})$ & $-\eta_{1^{1}3^{-2}5^{-2}15^{1}}$ & & $\II_{8,2}(3^{-3})$ & $9\eta_{1^{-9}3^{3}}$ \\
$\II_{4,2}(23^{-3})$ & $\eta_{1^{-1}23^{-1}}$ & & $\II_{8,2}(2_{\II}^{-8}3^{+3})$ & $\eta_{1^{-4}2^{-1}3^46^{-5}}$ \\
$\II_{4,2}(2_{\II}^{+4}3^{-3}5^{+3})$ & $\eta_{1^{-1}3^15^{1}6^{-1}10^{-1}15^{-1}}$ & & $\II_{8,2}(2_{\II}^{-8}3^{+7})$ & $\eta_{1^{4}2^{-5}3^{-4}6^{-1}}$ \\
$\II_{4,2}(2_{\II}^{+4}3^{+3}5^{-3})$ & $\eta_{1^{1}2^{-1}3^{-1}5^{-1}15^{1}30^{-1}}$ & & $\II_{8,2}(2_{\II}^{+2}3^{-7})$ & $\eta_{1^{-5}2^{4}3^{-1}6^{-4}}$ \\
$\II_{4,2}(2_{\II}^{+2}3^{+3}5^{-3})$ & $\eta_{2^{-1}3^{-1}5^{-1}6^{1}10^{1}30^{-1}}$ & & $\II_{8,2}(2_{\II}^{+2}3^{-3})$ & $\eta_{1^{-1}2^{-4}3^{-5}6^{4}}$ \\
$\II_{4,2}(2_{\II}^{+2}3^{-3}5^{+3})$ & $\eta_{1^{-1}2^{1}6^{-1}10^{-1}15^{-1}30^{1}}$ & & $\II_{8,2}(7^{-5})$ & $\eta_{1^{-3}7^{-3}}$ \\ \cline{1-2} \cline{4-5}
$\II_{6,2}(5^{+5})$ & $\eta_{1^{1}5^{-5}}$ & & $\II_{10,2}(2_{\II}^{+10})$ & $\eta_{1^{8}2^{-16}}$ \\
$\II_{6,2}(5^{+3})$ & $5\eta_{1^{-5}5^{1}}$ & & $\II_{10,2}(2_{\II}^{+2})$ & $16\eta_{1^{-16}2^{8}}$ \\
$\II_{6,2}(2_{\II}^{+6}3^{-4})$ & $\eta_{1^{2}2^{-4}3^26^{-4}}$ & & $\II_{10,2}(5^{+6})$ & $5\eta_{1^{-4}5^{-4}}$ \\
$\II_{6,2}(2_{\II}^{+2}3^{-4})$ & $4\eta_{1^{-4}2^{2}3^{-4}6^{2}}$ & & $\II_{10,2}(2_{\II}^{+6}3^{-6})$ & $\eta_{1^{-2}2^{-2}3^{-2}6^{-2}}$  \\ \cline{4-5}
$\II_{6,2}(2_{\II}^{-4}3^{-6})$ & $\eta_{1^{1}2^{1}3^{-3}6^{-3}}$ & & $\II_{14,2}(3^{-8})$ & $\eta_{1^{-6}3^{-6}}$ \\ \cline{4-5}
$\II_{6,2}(2_{\II}^{-4}3^{-2})$ & $3\eta_{1^{-3}2^{-3}3^{1}6^{1}}$   & & $\II_{18,2}(2_{\II}^{+10})$ & $\eta_{1^{-8}2^{-8}}$ \\ \cline{4-5}
$\II_{6,2}(2_{\II}^{-6}5^{-3})$ & $\eta_{1^{-2}2^{-1}5^{2}10^{-3}}$  & & $\II_{26,2}$ & $\eta_{1^{-24}}$
\end{tabular}
\end{table}
By working out the obstruction space for each of the finitely many lattices from Theorem \ref{thm1} we can now prove the following.
\begin{theorem}\label{thm2}
Let $L$ be an even lattice of signature $(n,2)$, $n\geq 4$ and squarefree level $N$ that splits $\II_{1,1}\oplus \II_{1,1}(N)$ and $\Psi$ a strongly reflective automorphic form of singular weight on the corresponding hermitian symmetric domain. Then $\Psi$ is the theta lift of $F_{\Gamma_0(N),f,0}$ with $f$ one of the functions from the table above. In particular, all $\Psi$ can be realized as the theta lift of a symmetric form $F$.
\end{theorem}
We remark that with the methods used in the present paper one could also obtain a complete list of reflective automorphic forms without the assumptions on its weight and on the order of its zeros, however, due to the number of cases it does not seem feasible to do this without a computer.

The paper is structured as follows: In Sections 2 and 3 we recall the necessary material on lattices, discriminant forms and the Weil representation. In Section 4 we recall Borcherds' singular theta correspondence, define reflective automorphic forms and prove Theorem \ref{thm1}. In Sections 5 and 6 we summarize how one can compute the Fourier expansions of certain cusp forms at various cusps. This is done for newforms for $\Gamma_0(N)$ and images of these in higher levels in Section 5 and for eta quotients for $\Gamma_1(N)$ in Section 6. We then use these expansions to prove Theorem \ref{thm2} in Section 7.

I thank Nils Scheithauer for helpful discussions on the content of this paper.

\section{Lattices and discriminant forms}
A finite abelian group $D$ with a $\Q/\Z$-valued non-degenerate quadratic form $D\to \Q/\Z, \gamma\mapsto \gamma^2/2$ is called a discriminant form. Every discriminant form decomposes into a sum of Jordan components (not uniquely if $p=2$) and every Jordan component is a sum of indecomposable Jordan components. The possible non-trivial Jordan components are as follows (for details consult \cite{ConwaySloane} or \cite{Nikulin}). 

Let $q>1$ be a power of an odd prime $p$. The non-trivial $p$-adic Jordan components of exponent $q$ are $q^{\pm n}$ for $n\geq 1$. The indecomposable Jordan components are $q^{\pm 1}$, generated by an element $\gamma$ with $q\gamma=0$ and $\gamma^2/2=a/q\mod 1$ where $a$ is an integer with $\legendre{2a}{q}=\pm 1$. The component $q^{\pm n}$ is a sum of $n$ copies of $q^{+1}$ and $q^{-1}$ where the number of copies of $q^{-1}$ is even if $\pm n = +n$ and odd if $\pm n = -n$. 

If $q>1$ is a power of $2$, then the even $2$-adic Jordan components of exponent $q$ are $q_{\II}^{\pm 2n}$ for $n\geq 1$. The indecomposable Jordan components are $q_{\II}^{\pm 2}$, which are generated by two elements $\gamma,\delta$ with $q\gamma=q\delta=0$, $\gamma\delta=1/q\mod 1$ and $\gamma^2/2=\delta^2/2=0\mod 1$ if $\pm=+$ and  $\gamma^2/2=\delta^2/2=1/q\mod 1$ if $\pm=-$. 
 
If $q>1$ is a power of $2$, then the odd $2$-adic Jordan components of exponent $q$ are $q_t^{\pm n}$ for $n\geq 1$ and $t\in\Z/8\Z$ with $\legendre{t}{2}=\pm 1$ if $n=1$, $t=-2,0,2$ if $n=2$ and $\pm= +$, $t=-2,2,4$ if $n=2$ and $\pm = -$, and $t=n\mod 2$ for any $n$. The indecomposable Jordan components are $q^{\pm 1}_t$ where $\legendre{t}{2}=\pm 1$, generated by an element $\gamma$ with $q\gamma=0$ and $\gamma^2/2=t/2q\mod 1$ (with some of them being isomorphic).  

The sum of two Jordan components with the same prime power $q$ can be determined by adding the ranks, multiplying the signs in the exponent and adding the subscripts $t$ if there are any.

The level of a discriminant form $D$ is the smallest positive integer $N$ with $N\gamma^2/2\in \Z$ for all $\gamma\in D$ and we define the signature $\sign(D)\in\Z/8\Z$ of $D$ by
\[
\sum_{\gamma\in D}e(\gamma^2/2)=\sqrt{|D|}e(\sign(D)/8),
\] 
where $e(z)=\exp(2\pi i z)$.

The signature of $q^{\pm n}$ for odd $q$ is given by $-n(q-1)+4k$ where $k=1$ if $q$ is not a square and $\pm n=-n$, otherwise $k=0$. For even $q$, the signature of $q_{\II}^{\pm 2n}$ is $4k$, whereas for $q_t^{\pm n}$ it is $t+4k$, where in both cases $k$ is as before.

For a discriminant form $D$ we define $D_c$ to be the elements $\gamma\in D$ with $c\gamma=0$, i.e. the elements of order dividing $c$. We define $D^c$ to be $\{\gamma\in D : \gamma=c\delta \text{ for some } \delta\in D\}$, i.e. the set of $c$-th powers in $D$. Then
\[
0\rightarrow D_c\rightarrow D\rightarrow D^c\rightarrow 0
\]
is exact and $D_c$ is the orthogonal complement of $D^c$. 

The number of elements of a given norm in the Jordan components of prime order have been determined by Scheithauer (assuming that $p^{\epsilon n}$ is even $2$-adic if $p=2$).

\begin{proposition}[\cite{ScheithauerClass}, Proposition 3.1]\label{prop:number2}
The number of elements in $2_{\II}^{\epsilon n}$ of norm $j/2\mod 1$ is
\[
N(2_{\II}^{\epsilon n},j)=\begin{cases} 2^{n-1}-\epsilon 2^{(n-2)/2} &\text{ if } j\neq 0, \\
2^{n-1}+\epsilon 2^{(n-2)/2} &\text{ if } j=0.\end{cases}
\]
\end{proposition}
\begin{proposition}[\cite{ScheithauerClass}, Proposition 3.2]\label{prop:numberp}
Let $p$ be an odd prime. Then the number of elements in $p^{\epsilon n}$ of norm $j/p\mod 1$ is given by
\[
N(p^{\epsilon n},j)=\begin{cases}p^{n-1}-\epsilon\legendre{-1}{p}^{n/2} p^{(n-2)/2} &\text{ if $n$ is even and } j\neq 0, \\
p^{n-1}+\epsilon\legendre{-1}{p}^{n/2}\left(p^{n/2}-p^{(n-2)/2}\right) &\text{ if $n$ is even and } j= 0, \\
p^{n-1}+\epsilon\legendre{-1}{p}^{(n-1)/2}\legendre{2}{p}\legendre{j}{p} p^{(n-1)/2} &\text{ if $n$ is odd and } j\neq 0, \\
p^{n-1} &\text{ if $n$ is odd and } j= 0.
\end{cases}
\]
\end{proposition}
Now suppose that $D$ has squarefree level $N$. Then the Jordan components of $D$ are of the form $p^{\epsilon_p n_p}$ for primes $p\mid N$ and there can be no odd $2$-adic Jordan components. Using the last two propositions one can then determine the numbers of elements of any given norm and order in $D$ as follows.
\begin{proposition}[\cite{ScheithauerClass}, Proposition 3.3]
Let $D$ be a discriminant form of squarefree level. Let $c\mid N$. Then the number of elements in $D_c$ of norm $j/c\mod 1$ is given by
\[
N(D_c,j)=\prod_{p\mid c}N(p^{\epsilon_p n_p},cj/p).
\]
\end{proposition}

If $L$ is an even lattice of signature $(b^+,b^-)$ with dual lattice $L'$, then $L'/L$ is a discriminant form with the quadratic form given by reducing the quadratic form on $L'$ modulo $1$ and conversely every discriminant form can be obtained in this way. The signature of $L'/L$ is $b^+-b^-\mod 8$ by Milgram's formula and we define the level of $L$ to be the level of $L'/L$.

For a positive integer $a$ we let $\II_{1,1}(a)$ be the lattice $\Z^2$ with Gram matrix $\begin{psmallmatrix} 0 & a \\ a & 0\end{psmallmatrix}$. If $a=1$, then we just write $\II_{1,1}$. Later we will only consider even lattices $L$ of squarefree level $N$ of the form $L=K\oplus\II_{1,1}\oplus \II_{1,1}(N)$ (i.e. $L$ splits $\II_{1,1}\oplus \II_{1,1}(N)$). We now give conditions on $D=L'/L$ for this to be satisfied.  
\begin{proposition}
Let $N$ be a squarefree positive integer, $L=\II_{1,1}(N)$ and $D=L'/L$. Then $D=\prod_{p\mid N} p^{\epsilon_pn_p}$ with $n_p=2$ and $\epsilon_p=\legendre{-1}{p}$ for all $p$.
\end{proposition}
\begin{proof}
Note that $L'=\left(\frac{1}{N}\Z\right)^2$. Therefore $L'/L=\left(\Z/N\Z\right)^2$ which proves that $n_p=2$ for all $p\mid N$. We fix a prime $p\mid N$. Then $\gamma=(1/p,0)+L\in D$ is an isotropic element of order $p$. In particular $D_p=p^{\epsilon_p2}$ has nontrivial isotropic elements. By Propositions \ref{prop:number2} and \ref{prop:numberp} this is only possible if $\epsilon_p$ is as claimed.
\end{proof}
\begin{proposition}
Let $L$ be an even lattice of signature $(b^+,b^-)$ and squarefree level $N$ and $D=L'/L$. Suppose that $L$ splits $\II_{1,1}$. Then $D=\prod_{p|N}p^{\epsilon_pn_p}$ with $n_p\leq b^++b^--2$. If $n_p=b^++b^--2$, then $\epsilon_p=\legendre{|D^p|}{p}\legendre{-1}{p}^{b^--1}$.
\end{proposition}
\begin{proof}
We can write $L=K\oplus \II_{1,1}$ for an even lattice $K$ of signature $(b^+-1,b^--1)$. Then  $K'/K\cong L'/L$ and Theorem 1.10.1 in \cite{Nikulin} applied to $K$ shows that $n_p\leq b^++b^--2$ and that $\epsilon_p$ must be as claimed if this is an equality. 
\end{proof}
It follows that if $D=L'/L=\prod_{p\mid N}p^{\epsilon_p n_p}$ for an even lattice $L$ of signature $(b^+,b^-)$ and squarefree level $N$ that splits $\II_{1,1}\oplus \II_{1,1}(N)$, then $2\leq n_p\leq b^++b^--2$ for all $p\mid N$ and that 
\[
\epsilon_p=\begin{cases}\legendre{-1}{p} &\text{ if } n_p=2, \\
						\legendre{|D^p|}{p}\legendre{-1}{p}^{b^--1} &\text{ if }n_p=b^++b^--2.\end{cases}
\]
	
A primitive vector $\alpha\in L$ of positive norm is called a root if the reflection 
\[\sigma_\alpha(x)=x-2\frac{(x,\alpha)}{(\alpha,\alpha)}\alpha
\]
at $\alpha^\perp$ is an automorphism of $L$. The following proposition describes the roots of $L$ if the level $N$ is squarefree.
\begin{proposition}[see \cite{ScheithauerClass}, Propositions 2.1 and 2.2]\label{prop:roots}
Let $L$ be an even lattice of squarefree level $N$ and let $\alpha\in L$. Then $\alpha$ is a root if and only if $k=\alpha^2/2$ is a positive divisor of $N$ and $\alpha\in L\cap kL'$.
\end{proposition}
If the level of $L$ is squarefree and $\gamma\in D=L'/L$, then we say that $\gamma$ corresponds to roots if $\gamma$ has order $k$ and satisfies $\gamma^2/2=1/k \mod 1$ for some $k\mid N$. This notion is inspired by Proposition 2.5 in \cite{ScheithauerClass}.

\section{The Weil representation}
Let $D$ be a discriminant form with even signature. There is a unitary group action of $\SL_2(\Z)$ on the group ring $\C[D]$, defined by
\begin{align*}
\rho_D(T)\e_\gamma&=e(-\gamma^2/2)\e_\gamma \\
\rho_D(S)\e_\gamma&=\frac{e(\sign(D)/8)}{\sqrt{|D|}}\sum_{\beta\in D}e(\beta\gamma)\e_\beta,
\end{align*}
where $S=\begin{psmallmatrix} 0 & -1 \\ 1 & 0\end{psmallmatrix}$ and $T=\begin{psmallmatrix} 1 & 1 \\ 0 & 1\end{psmallmatrix}$ are the standard generators of $\SL_2(\Z)$. This is called the Weil representation of $\SL_2(\Z)$. 

A holomorphic function $F\colon \H \to \C[D]$ on the upper half-plane $\H=\{z\in\C : \Im(z)>0\}$ is called a weakly holomorphic modular form of weight $k\in\Z$ for $\rho_D$ if 
\[
F(M\tau)=(c\tau+d)^k\rho_D(M)F(\tau)
\]
for all $M=\begin{psmallmatrix} a & b \\ c & d\end{psmallmatrix}\in\SL_2(\Z)$ and $F$ is meromorphic at the cusp $\infty$. We denote by $M_{k}^!(\rho_D)$ the complex vector space of such functions. An element of $M_{k}^!(\rho_D)$ is called holomorphic if it is holomorphic at $\infty$ and is called cusp form if it vanishes at $\infty$. The spaces of holomorphic modular forms and cusp forms are denoted by $M_{k}(\rho_D)$ and $S_{k}(\rho_D)$.

The Weil representation can be computed explicitly; see \cite{ScheithauerWeil}, Theorem 4.7. From the resulting formulas one can see that the components $F_\gamma$ of a weakly holomorphic modular form $F=\sum_{\gamma\in D} F_\gamma\e_\gamma$ of weight $k$ for $\rho_D$ transform as
\begin{equation}\label{eqn:ComponentsTransformation}
F_\gamma|_kM=\xi(M)\frac{\sqrt{|D_c|}}{\sqrt{|D|}}\sum_{\beta\in D^{c*}}e(-d\beta_c^2/2)e(-b\beta\gamma)e(-ab\gamma^2/2)F_{a\gamma+\beta},
\end{equation}
where  $|_k$ is the Petersson slash operator, $M=\begin{psmallmatrix} a & b \\ c & d\end{psmallmatrix}\in\SL_2(\Z)$, $\xi$ is the root of unity as in \cite{ScheithauerWeil}, Theorem 4.7, and $D^{c*}$ and $\beta_c^2/2$ are as in Section 2 of the same article. If $N$ is squarefree, then
\[
\xi(M)=e(\sign(D)/8)\legendre{d}{|D_c|}\legendre{c}{|D^c|}\prod_{p\mid c}{e(-\sign(D_p)/8)}
\]
and $D^{c*}=D^c$. In particular, $F_\gamma$ is a weakly holomorphic modular form of weight $k$ and character $\chi_\gamma$ with $\chi_\gamma(b)=e(-b\gamma^2/2)$ for $\Gamma_1(N)$. For any scalar-valued modular form $f$ and rational number $x\in\Q$ we write $[f](x)$ for the Fourier coefficient of $f$ at $q^x$, so in particular $[F_\gamma](x)$ denotes the respective Fourier coefficient of $F_\gamma$.

One can construct modular forms for $\rho_D$ from scalar-valued modular forms as follows.
\begin{proposition}[\cite{ScheithauerModForms}, Theorem 3.1]\label{prop:lift}
Let $k$ be an integer, $N$ the level of $D$ and $f$ a weakly holomorphic modular form of weight $k$ for $\Gamma_0(N)$ with character $\chi_D$ where
\[
\chi_D(M)=\legendre{a}{|D|}e((a-1)\oddity(D)/8).
\]
Let $H$ be an isotropic subset of $D$ which is invariant under $(\Z/N\Z)^*$ as a set. Then
\[
F_{\Gamma_0(N),f,H}=\sum_{M\in\Gamma_0(N)\backslash\SL_2(\Z)}\sum_{\gamma\in H}f|_kM\rho_D(M^{-1})\e_\gamma
\]
is an element of $M_k^!(\rho_D)$ which is invariant under the automorphisms of $D$ that stabilize $H$ as a set. Moreover, if $f$ is a holomorphic modular form (resp. a cusp form) then so is $F_{\Gamma_0(N),f,H}$.
\end{proposition}
We will only need this for $H=\{0\}$ and squarefree $N$. We summarize Section 6 of \cite{ScheithauerClass} to show how the lift $F_{\Gamma_0(N),f,0}$ can be computed in this case. For each positive divisor $c$ of the squarefree level $N$ we choose a matrix
\[
M_c=\begin{pmatrix}
1 & b \\ c & d
\end{pmatrix}\in\SL_2(\Z)
\] with $d=1\mod c$ and $d=0\mod c'$ where $c'=N/c$ and let 
\[
f_c=f|_kM_c.
\]
The function $f_c$ has a Fourier expansion in integral powers of $q^{1/c'}$ because the cusp $1/c$ of $\Gamma_0(N)$ has width $c'$, so we can write
\[
f_c=g_{c',0}+g_{c',1}+\cdots+g_{c',c'-1}
\]
with
\[
g_{c',j}|_T(\tau)=e(j/c')g_{c',j}(\tau).
\]
Then
\[
F_{\Gamma_0(N),f,0}(\tau)=\sum_{c|N}\sum_{\mu\in D_{c'}}\xi_c\frac{\sqrt{|D_c|}}{\sqrt{|D|}}c'g_{c',j_{\mu,c'}}(\tau)\e_\mu
\]
where $j_{\mu,c'}/c'=-\mu^2/2\mod 1$ and 
\begin{align*}
\xi_c&=e(\sign(D)/8)\legendre{-c}{|D_{c'}|}\prod_{p\mid c}e(-\sign(D_p)/8)\\
&=\legendre{-c}{|D_{c'}|}\prod_{p\mid c'}e(\sign(D_p)/8).
\end{align*}
Note that $F_{\Gamma_0(N),f,0}$ is symmetric, i.e. invariant under the automorphisms of $D$, by Proposition \ref{prop:lift}. Since $N$ is squarefree, the converse also holds, i.e. every weakly holomorphic symmetric modular form $F$ is equal to $F_{\Gamma_0(N),f,0}$ for some weakly holomorphic $f$ (see \cite{ScheithauerModForms}, Corollary 5.5).

Let $D$ be a discriminant form and $H\subset D$ an isotropic subgroup. Then $D_H=H^\perp/H$ is a discriminant form whose signature equals that of $D$. There is a map from $M_k^!(\rho_{D_H})$ to $M_k^!(\rho_D)$, given by
\[
F=\sum_{\gamma+H\in D_H}F_{\gamma+H}\e^{\gamma+H}\mapsto \hat{F}=\sum_{\gamma\in D}F_{\gamma+H}\e^\gamma.
\]  
This map sends holomorphic modular forms to holomorphic modular forms and cusp forms to cusp forms. We say that $\hat{F}$ is the lift of $F$ on $H$.

\section{Reflective automorphic forms}
Let $L$ be an even lattice of signature $(n,2)$ with $n\geq 4$ even, $V=L\otimes_\Z\R$ and $V(\C)=V\otimes_\R\C$. Then
\[
\{[Z]\in \P(V(\C)) : (Z,Z)=0, (Z,\overline{Z})<0\}
\]
is a complex manifold with two connected components, which are exchanged by $Z\mapsto \overline{Z}$. We choose one of these components and denote it by $\mathcal{H}_L$. The subgroup $O(V)^+\subset O(V)$ that preserves the two connected components acts holomorphically on $\mathcal{H}_L$. We let 
\[
\tilde{\mathcal{H}}_L=\{Z\in V(\C)\setminus\{0\} : [Z]\in \mathcal{H}_L\}
\] 
be the affine cone over $\mathcal{H}_L$.

Let $\Gamma\subset O(L)^+=O(L)\cap O(V)^+$ be a subgroup of finite index, $\chi\colon \Gamma\to \C^*$ a unitary character and $k$ an integer. A meromorphic function $\Psi\colon \tilde{\mathcal{H}}_L\to\C$ is an automorphic form of weight $k$ for $\Gamma$ and $\chi$ if
\begin{align*}
\Psi(MZ)&=\chi(M)\Psi(Z) \\
\Psi(tZ)&=t^{-k}\Psi(Z)
\end{align*}
for all $M\in\Gamma$, $t\in\C^*$ and $Z\in\tilde{\mathcal{H}}_L$. One way to obtain such functions is described in the following theorem. 
\begin{theorem}[\cite{BorcherdsGrass}, Theorem 13.3]\label{thm:133}
Let $L$ be an even lattice of type $(n,2)$, $n\geq 3$ with $D=L'/L$ and $F$ a weakly holomorphic modular form of weight $1-n/2$ for $\rho_D$ with integral coefficients $[F_\gamma](m)$ for all $m\leq 0$. Then there is a meromorphic function $\Psi\colon\tilde{\mathcal{H}}_L\to \C$ with the following properties.
\begin{enumerate}
\item $\Psi$ is a modular form of weight $[F_0](0)/2$ for the group $O(L,F)^+$ and some unitary character $\chi$.
\item  The only zeros or poles of $\Psi$ are on rational quadratic divisors $\lambda^\perp$ for $\lambda\in L$ of positive norm and are zeros of order
\begin{equation}\label{DivisorAutProd}
\sum_{\substack{0<x\in\Q \\ x\lambda\in L'}} [F_{x\lambda+L}](-x^2\lambda^2/2)
\end{equation}
or poles if this number is negative.
\item For each primitive norm $0$ vector $z\in L$ and for each Weyl chamber $W$ of $K=M/\Z z$ with $M=L\cap z^\perp$ the restriction $\Psi_z$ has an infinite product expansion converging when $Z$ is in a neighbourhood of the cusp $z$ which is some constant times
\[
e((Z,\rho(K,W,F_K)))\prod_{\substack{\lambda\in K'\\(\lambda,W)<0}}\prod_{\substack{\delta\in L'/L \\ \delta|M=\lambda}}(1-e((\lambda,Z)+(\delta,z')))^{[F_\delta](-\lambda^2/2)}.
\] 
\end{enumerate}
\end{theorem}
The function $\Psi$ is called the automorphic product corresponding to $F$. Bruinier proved the following converse theorem.
\begin{theorem}[\cite{BruinierConverse}, Theorem 1.2]\label{thm:converse}
Let $L$ be an even lattice of signature $(n,2), n\geq 4$ even such that $L=K\oplus \II_{1,1}\oplus \II_{1,1}(m)$ for an even positive definite lattice $K$ and some positive integer $m$. Then every automorphic form $\Psi$ for the discriminant kernel $\Gamma(L)$ of $O(L)^+$ whose divisor is a linear combination of rational quadratic divisors is (up to a nonzero constant factor) the Borcherds lift of a weakly holomorphic vector-valued modular form for the Weil representation.
\end{theorem}

Different vector-valued modular forms $F$, even on different lattices, can give the same automorphic product $\Psi$. We describe an example of when this occurs. Suppose $L=K+\II_{1,1}(m)$ for some positive integer $m$ and let $M\subset K$ be a sublattice of finite index. Then $H=K/M\subset K'/M\subset M'/M$ is an isotropic subgroup of $D_M=M'/M$ with orthogonal complement $H^\perp=K'/M$. The quotient $H^\perp/H$ is isomorphic to $K'/K$ and a modular form $F_L$ as in Theorem \ref{thm:133} on $D_L=L'/L$ induces a modular form $F_N$ on $D_N=N'/N$ where $N=M\oplus \II_{1,1}(m)$. The embedding $N\hookrightarrow L$ gives an identification of the corresponding domains $\mathcal{H}_N$ and $\mathcal{H}_L$ and the automorphic products corresponding to $F_L$ and $F_N$ coincide (see \cite{ScheithauerSingular}, Proposition 3.4).

We say that an automorphic form $\Psi$ is reflective if it is holomorphic and all its zeros are of the form $\lambda^\perp$ for roots $\lambda\in L$. If in addition all zeros are simple, then we say that $\Psi$ is strongly reflective. If $L$ splits $\II_{1,1}$, then (strong) reflectivity of an automorphic product $\Psi$ can easily be checked on $F$. 
\begin{proposition}\label{prop:reflectiveProduct}
Suppose $L$ has squarefree level and splits $\II_{1,1}$. Then the automorphic product $\Psi$ is reflective if and only if the corresponding vector-valued modular form $F$ satisfies the following:
\begin{enumerate}
\item If $\gamma\in D$ has order $m$ and corresponds to roots, then the Fourier expansion of $F_\gamma$ at $\infty$ is $F_\gamma=c_{\gamma,-1/m}q^{-1/m}+O(1)$ with $c_{\gamma,-1/m}\geq 0$ and
\item  $F_\gamma$ is holomorphic at $\infty$ for all other $\gamma\in D$.
\end{enumerate}
Moreover, $\Psi$ is strongly reflective if and only if all $c_{\gamma,-1/m}$ are at most $1$.
\end{proposition}
\begin{proof}
Suppose that $F$ satisfies the two conditions. Then Propositions 9.1, 9.2 and 9.3 in \cite{ScheithauerClass} prove that $\Psi$ is reflective. It also follows that $\Psi$ is strongly reflective if all $c_{\gamma,-1/m}$ are at most $1$.

Now assume that $\Psi$ is reflective. Write $L=K\oplus\II_{1,1}$. Then $L'=K'\oplus \II_{1,1}$ and $D=L'/L\cong K'/K$. Let $\gamma\in D$ and $x<0$ such that $[F_\gamma](x)\neq 0$ (so $\gamma^2/2=-x\mod 1$) and let $m$ be the largest integer with $[F_{m\gamma}](m^2x)\neq 0$. Choose $\kappa\in K'$ with $\kappa+K=m\gamma$. By adding a primitive element of suitable norm in $\II_{1,1}$ to $\kappa$ we obtain a primitive element $\lambda\in L'$ with $\lambda^2/2=-m^2x$. Then $M\lambda\in L$ is primitive, where $M$ is the order of $m\gamma\in D$ and the order of $\Psi$ at $(M\lambda)^\perp$ is equal to $[F_{m\gamma}](m^2x)$. It follows that $M\lambda$ is a root of $L$ (and that $[F_{m\gamma}](m^2x)=1$ if $\Psi$ is strongly reflective). Therefore the norm $(M\lambda)^2/2=-M^2m^2x\in\Z$ of $M\lambda$ must divide $N$. Then $x=-a/(Mm)$ for some $a\in \Z_{>0}$ because $N$, and hence also the denominator of $x$, are squarefree. Moreover $\lambda/(Mm^2x)$ must be in $L'$ by Proposition \ref{prop:roots}, so $1/(Mm^2x)\in\Z$ by the primitivity of $\lambda$. This forces $a=m=1$, so $\gamma$ satisfies the condition given in item (1). 

\end{proof}
If $L$ is an even lattice of squarefree level and $F$ is a vector-valued modular form of weight $1-n/2$ on $L'/L$  that satisfies the conditions (1) and (2), then we also call $F$ reflective (strongly reflective if in addition all $c_{\gamma,-1/m}$ are at most 1), so that the statement of the proposition can be rephrased as follows: \emph{If $L$ splits $\II_{1,1}$, then $\Psi$ is (strongly) reflective if and only if $F$ is (strongly) reflective.} If instead $F$ only satisfies conditions (1) and (2) with $c_{\gamma,-1/m}$ being any complex number, then we say that $F$ is semi-reflective. Note that, in contrast to the reflective modular forms, the semi-reflective modular forms form a complex vector space. 

For the rest of the section we assume that $L$ is an even lattice of squarefree level $N$ and signature $(n,2)$ with $n\geq 4$ such that $L$ splits $\II_{1,1}(N)$ and that $F$ is a semi-reflective form with $F_0\neq 0$. We replace $F$ by 
\begin{equation}\label{eqn:symmetrization}
\frac{1}{|\Aut(D)|}\sum_{\sigma\in\Aut(D)}\sigma(F),
\end{equation}
 which is symmetric and nonvanishing (because $F_0\neq 0$). Then $F=F_{\Gamma_0(N),f,0}$ for a nonzero weakly holomorphic modular form $f$ of weight $k=1-n/2$ for $\Gamma_0(N)$ with character $\chi_D$ by Corollary 5.5 in \cite{ScheithauerClass}. 

\begin{lemma}\label{lem:polesOfScalarForm}
Let $c\mid N$. Then $f|_kM_c\in O(q^{-1/c'})$, where $c'=N/c$.
\end{lemma}
\begin{proof}
Suppose there is some cusp $s=1/c$ with $c\mid N$ such that $f$ has a pole of order larger than $1/c'$ at $s$, i.e. $f|_kM_c\notin O(q^{-1/c'})$. We can assume that $c$ is the smallest divisor of $N$ with this property, so $f|_kM_{\tilde{c}}\in O(q^{-1/\tilde{c}'})$ for all $\tilde{c}\mid N$ with $\tilde{c}<c$. For any $d\mid N$ and $a\in \Z$ the discriminant form corresponding to $\II_{1,1}(N)$ contains elements of order $d$ and norm $a/d\mod 1$. Since $L$ splits $\II_{1,1}(N)$, there is therefore an element $\gamma\in D=L'/L$ of order $c'$ and norm $\gamma^2/2=a/c'\mod 1$, where $a/c'$ is the order of the pole of $f$ at $1/c$. Then
\[
F_\gamma=\sum_{d\mid c}\xi_d\frac{\sqrt{|D_d|}}{\sqrt{|D|}}d'g_{d',j_{\gamma,d'}}=\xi_c\frac{\sqrt{|D_c|}}{\sqrt{|D|}}c'g_{c',j_{\gamma,c'}}+\sum_{\substack{d\mid c\\ d<c}}\xi_d\frac{\sqrt{|D_d|}}{\sqrt{|D|}}d'g_{d',j_{\gamma,d'}}.
\] 
The first summand on the right hand side has a pole of order $a/c'$, while the remaining terms have poles of order less than $1/c'$ by our assumption on the minimality of $c$. Therefore $F_\gamma$ must have a pole of order $a/c'>1/c'$ which contradicts the assumption that $F$ is semi-reflective.  
\end{proof}

\begin{lemma}\label{lem:bounds}
Let $N$ be a squarefree integer, $k$ a negative integer and $f$ a nonzero weakly holomorphic modular form of weight $k$ and some Dirichlet character $\chi$ for  $\Gamma_0(N)$ satisfying $f|_kM_c=O(q^{-1/c'})$ for all $c\mid N$. Then
\[
\prod_{p\mid N}(p+1)\leq -2^{\omega(N)}\frac{12}{k}
\] 
where $\omega(N)$ is the number of primes dividing $N$.
\end{lemma}
\begin{proof}
This follows immediately from the valence formula for $\Gamma_0(N)$ (see e.g. Theorem 4.1 in appendix I of \cite{Hirzjung}), because the left hand side is the index of $\Gamma_0(N)$ in $\SL_2(\Z)$, $2^{\omega(N)}$ is the number of cusps of $\Gamma_0(N)$ and the width of the cusp $1/c$ is $c'$.
\end{proof}
This can only be satisfied if $k$ is at least $-12$ and $\omega(N)$ is at most $3$. For fixed $k$ and $\omega(N)$, the level $N$ is then bounded by the values in the following table (with \enquote{-} meaning that the set of possible $N$ is empty).
\begin{table}[!h]
\caption{Bound on $N$ depending on $k$ and $\omega(N)$}
\label{tab1}
\begin{tabular}{cc*{12}{c}}
& & \multicolumn{12}{c}{$k$} \\
& \multicolumn{1}{c|}{}& $-1$ & $-2$ & $-3$ & $-4$ & $-5$ & $-6$ & $-7$ & $-8$ & $-9$ & $-10$ & $-11$ & $-12$\\
\cline{2-14}
\multirow{4}{*}{$\omega(N)$} & \multicolumn{1}{c|}{0} &  1 & 1 & 1 & 1 & 1 & 1 & 1 & 1 & 1 & 1 & 1 & 1\\
&\multicolumn{1}{c|}{1} & 23 & 11 & 7 & 5 & 3 & 3 & 2 & 2 & - & - & - & - \\
&\multicolumn{1}{c|}{2} & 35 & 15 & 6 & 6 & - & - & - & - & - & - & - & - \\
&\multicolumn{1}{c|}{3} & 42 & - &  - & - & - & - & - & - & - & - & - & -
\end{tabular}
\end{table}
\begin{proof}[Proof of Theorem \ref{thm1}]
Suppose $\Psi$ is a nonconstant reflective automorphic form for $\Gamma(L)$. By Theorem \ref{thm:converse}, we can assume that $\Psi$ is the theta lift of a modular form $F$ on $D=L'/L$, which must be reflective by Proposition \ref{prop:reflectiveProduct}. Note that the weight of $\Psi$ is nonzero because $\Psi$ is not constant. Therefore $[F_0](0)\neq 0$ and hence $F_0\neq 0$. We can therefore find a weakly holomorphic modular form $f\neq 0$ of weight $k=1-n/2$ and character $\chi_D$ for $\Gamma_0(N)$ such that the symmetrization \eqref{eqn:symmetrization} of $F$ is $F_{\Gamma_0(N),f,0}$ and apply Lemmas \ref{lem:polesOfScalarForm} and \ref{lem:bounds}. This shows that $k$ and $N$ are bounded as above. In particular, there are only finitely many possibilities for $n$ and $N$. The result follows because there are only finitely many lattices of a fixed rank $n+2$ and level $N$.  
\end{proof}

\section{Fourier expansions of newforms}
Let $N$ be a squarefree positive integer, $k$ any integer and $\chi$ a Dirichlet character of modulus $N$. For a divisor $c\mid N$ we let $c'=N/c$ and write $\chi_c$ and $\chi_{c'}$ for the unique characters of modulus $c$ and $c'$ such that $\chi=\chi_c\cdot\chi_{c'}$. Suppose 
\[
g(\tau)=\sum_{n=1}^\infty a_n q^n
\]
is a newform, i.e. a normalized eigenform for all Hecke operators, of weight $k$ and character $\chi$ for $\Gamma_0(N)$. For each divisor $c\mid N$ we want to compute the Fourier expansion of $g|_kM_c$ at $\infty$, where $M_c$ is as in Section 3. This expansion will be an expansion in powers of $q^{1/c'}$ because the cusp $1/c$ of $\Gamma_0(N)$ has width $t_c=c'$. 

To compute these Fourier expansions we choose integers $\lambda_1$ and $\lambda_2$ such that
\[
W_{c'}=\begin{pmatrix}\lambda_1c' & \lambda_2 \\ -N & c'\end{pmatrix}
\]
has determinant $c'$. By Atkin-Lehner theory $g|_kW_{c'}$ is a newform of weight $k$ for $\Gamma_0(N)$ but with character $\overline{\chi}$. Its Fourier expansion at $\infty$ can be calculated with the following proposition.
\begin{proposition}[see \cite{Asai}, Section 1]\label{prop:Asai}
Let $N$ be squarefree and let $g$ be a newform of weight $k$ and character $\chi$ for $\Gamma_0(N)$. Let $c\mid N$ and $c'=N/c$. Then the Fourier expansion of $g|_kW_{c'}$ is given by
\[
g|_kW_{c'}=\lambda \sum_{n=1}^\infty a_n^{(c')}q^n
\]
where 
\[
\begin{cases}a_n^{(c')}=\overline{\chi}_{c'}(n)a_n &\text{ if } (n,c')=1, \\
a_n^{(c')}=\chi_c(n)\overline{a_n} &\text{ if } (n,c)=1, \\
a_{nm}^{(c')}=a_n^{(c')}a_m^{(c')} &\text{ if } (n,m)=1,\end{cases}
\]
and 
\[
\lambda = \chi_{c'}(c)\chi_c(c')\prod_{p\mid c'}\chi_p(c'/p)\lambda_p,
\]
with 
\[
\lambda_p=\begin{cases}G(\chi_p)q^{-k/2}\overline{a_p} &\text{ if } \chi_p \text{ is primitive, } \\
-q^{1-k/2}\overline{a_p} &\text{ if } \chi_p \text{ is principal. }\end{cases}
\]
Here $G(\chi_p)=\sum_{h=1}^{p-1}\chi_p(h)e(h/p)$ is the usual Gauss sum.
\end{proposition}
The Fourier expansions of $g|_kM_c$ and $g|_kW_{c'}$ are related as follows.
\begin{proposition}\label{prop:expansionNewforms}
Let $g$ be a newform of weight $k$ and character $\chi$ for $\Gamma_0(N)$ with $N$ squarefree. Then
\[
g|_kM_c(\tau)=\chi_{c'}(-1)\chi_c^{-1}(c')c'^{-k/2}g|_kW_{c'}(\tau/c').
\]
\end{proposition}
\begin{proof}
Let 
\[
A=\begin{pmatrix}c'+bc & \lambda_1b-\lambda_2\\ N+cd & \lambda_1d-\lambda_2c\end{pmatrix}.
\]
Then $A\in \Gamma_0(N)$ and
\[
M_c=\frac{1}{c'}AW_{c'}\begin{pmatrix}
1 & 0 \\ 0 & c'
\end{pmatrix}.
\]
This implies that
\[
g|_kM_c(\tau)=\chi(\lambda_1d-\lambda_2c)c'^{-k/2}g|_kW_{c'}(\tau/c').
\]
But 
\[
\chi(\lambda_1d-\lambda_2c)=\chi_{c'}(-\lambda_2c)\chi_c(\lambda_1d)
\]
since $c'|d$. From $\det(W_{c'})=c'$ we see that $\lambda_1c'+\lambda_2c=1$, so that $\lambda_1=c'^{-1}\mod c$ and $\lambda_2=c^{-1}\mod c'$. Together with $d=1\mod c$ this implies that 
\[
\chi(\lambda_1d-\lambda_2c)=\chi_{c'}(-1)\chi_c^{-1}(c')
\]
which completes the proof.
\end{proof}

Next we suppose that $g$ is a newform of some smaller level $M\mid N$ and character $\chi$ of modulus $M$. We have just seen how to compute the Fourier expansions of $g|_kM_c$ for $c\mid M$. For later use we also want to compute the Fourier expansions of $g|_kM_c$ for $c\mid N$, as well as those of $h|_kM_c$ where 
\[
h(\tau)=g|_k\begin{pmatrix}
N/M & 0 \\ 0 & 1
\end{pmatrix}(\tau)=(N/M)^{-k/2}g(N\tau/M),
\]
which is also a cusp form of level $N$ and character $\chi$.
\begin{proposition}
Let $N$ be squarefree and let $g$ be a newform of weight $k$ and character $\chi$ for $\Gamma_0(M)$ for some $M\mid N$. Let $c\mid N$,  define $m=(c,M)$ and choose a matrix
\[
\tilde{M}_m=\begin{pmatrix}1 & x \\ m & y\end{pmatrix}\in\SL_2(\Z)
\]
with $y=0\mod M/m$. Then
\[
g|_kM_c=\chi_{M/m}(c)\chi_{M/m}^{-1}(m)g|_k\tilde{M}_m
\]
\end{proposition}
\begin{proof}
Note that
\[
M_c=\begin{pmatrix}
1 & b \\ c & d
\end{pmatrix}=\begin{pmatrix}
y-bm & -x+b \\ cy-dm & -cx+d
\end{pmatrix}\begin{pmatrix}
1 & x \\ m & y
\end{pmatrix}
\]
and that the first matrix on the right hand side is in $\Gamma_0(M)$. Therefore
\[
g|_kM_c=\chi(-cx+d)g|_k\tilde{M}_m.
\]
To compute $\chi(-cx+d)$ we decompose $\chi$ as $\chi_m\cdot\chi_{M/m}$ and obtain
\[
\chi(-cx+d)=\chi_{M/m}(-cx)
\]
where we have also used that $m\mid c$, $d=1\mod m$ and $(M/m)\mid d$. From $y-mx=1$ and $y=0\mod M/m$ we infer that $mx=-1\mod M/m$. Hence
\[
\chi(-cx+d)=\chi_{M/m}(-cx)=\chi_{M/m}(c)\chi_{M/m}^{-1}(m).
\]
\end{proof}
\begin{proposition}
Let $N, M, c, m$ and $g$ be as in the previous proposition. Let 
\[
h(\tau)=g|_k\begin{pmatrix}
N/M & 0 \\ 0 & 1
\end{pmatrix}(\tau)
\]
and let $r_1=(c,N/M)$ and $r_2=(c',N/M)$. Choose a matrix
\[
\tilde{M}_m=\begin{pmatrix}
1 & x \\ m & y
\end{pmatrix}\in\SL_2(\Z)
\] with $y=0\mod M/m$. Then 
\[
h|_kM_c(\tau)=\chi_m^{-1}(r_2)(N/M)^{-k/2}g|_k\tilde{M}_m(r_1\tau/r_2).
\]
\end{proposition}
\begin{proof}
Note that
\begin{align*}
\begin{pmatrix}
N/M & 0 \\ 0 & 1
\end{pmatrix}M_c&=\begin{pmatrix}
N/M  & 0 \\ 0 & 1
\end{pmatrix}\begin{pmatrix}
1 & b \\ c & d
\end{pmatrix}\\
&=\begin{pmatrix}
yr_2-bmr_1 & br_1-xr_2 \\ my-md/r_2 & d/r_2-mx
\end{pmatrix}\begin{pmatrix}
1 & x \\ m & y
\end{pmatrix}\begin{pmatrix}
r_1 & 0 \\ 0 & r_2
\end{pmatrix}
\end{align*}
and that the first matrix of the last line is in $\Gamma_0(M)$. Therefore
\[
h|_kM_c(\tau)=\chi(d/r_2-mx)(N/M)^{-k/2}g|_k\tilde{M}_m(r_1\tau/r_2).
\]
It remains to compute $\chi(d/r_2-mx)$. But this is
\[
\chi(d/r_2-mx)=\chi_m(d/r_2)\chi_{M/m}(-mx).
\]
Moreover, $\chi_m(d)=1$ because $d=1\mod c$ and $mx=y-1$, so $\chi_{M/m}(-mx)=1$ which completes the proof.
\end{proof}
\section{Eta quotients}
The Dedekind eta function is the holomorphic function on the upper half-plane defined by
\[
\eta(\tau)=q^{1/24}\prod_{n>0}(1-q^n).
\]
It transforms as follows.
\begin{proposition}[\cite{Rademacher}, p.163]\label{prop:etaTransformation}
Let $M=\begin{psmallmatrix} a & b \\ c & d\end{psmallmatrix}\in \SL_2(\Z)$ with $c>0$. Then
\[
\eta(M\tau)=\varepsilon(M)\sqrt{c\tau+d}\eta(\tau)
\]
where
\[
\varepsilon(M)=\begin{cases}
\legendre{d}{c}e((-3c+bd(1-c^2)+c(a+d))/24) &\text{ if } c \text{ is odd, } \\
\legendre{c}{d}e((3d-3+ac(1-d^2)+d(b-c))/24) &\text{ if } c \text{ is even. }
\end{cases}
\]
\end{proposition}
Next we let $\eta_k(\tau)=\eta(k\tau)$ for  a positive integer $k$.
\begin{proposition}[\cite{ScheithauerWeil}, Proposition 6.2]\label{prop:rescaledEta}
Let $M=\begin{psmallmatrix} a & b \\ c & d\end{psmallmatrix}\in \SL_2(\Z)$ with $c>0$. Let $r,s,t\in\Z$ with $r,t>0$ and
\[
rt=k, \quad r\mid c, \quad k\mid(dr-cs). 
\]
Then
\[
\eta_k(M\tau)=\varepsilon\left(\begin{pmatrix}
at & br-as \\ c/r & (dr-cs)/k
\end{pmatrix}\right)\frac{1}{\sqrt{t}}\sqrt{c\tau+d}\eta\left(\frac{r\tau+s}{t}\right).
\]
\end{proposition}
\begin{proposition}[\cite{DHS}, Proposition 5.1]\label{prop:DHS}
Let $N$ be a positive integer. Take integers $r_\delta$ for $\delta\mid N$ such that $\frac{N}{24}\sum_{\delta\mid N}\delta r_\delta$ and $\frac{N}{24}\sum_{\delta\mid N} r_\delta/\delta$ are integers and $\sum_{\delta\mid N} r_\delta$ is even. Then the eta quotient
\[
\prod_{\delta\mid N} \eta(\delta\tau)^{r_\delta}
\]
is a weakly holomorphic modular form for $\Gamma_1(N)$ of weight $k=\sum_{\delta\mid N} r_\delta/2$ and character 

\[
\chi\left(\begin{pmatrix}
a & b \\ c & d 
\end{pmatrix}
\right)=e\left(\frac{b}{24}\sum_{\delta\mid N}\delta r_\delta\right).
\]
\end{proposition}
\begin{proposition}\label{prop:etaCuspForm}
Suppose $N$ is squarefree. Then the eta quotient from the last proposition is a cusp form if and only if 
\[
\sum_{\delta\mid N} \frac{(\delta,c)}{(\delta,c')} r_\delta>0
\]
for all $c\mid N$.
\end{proposition}
\begin{proof}
Let $a/c\in \Q\cup\infty$ be a cusp and let $c'=N/(N,c)$. We can suppose that $c>0$, because every cusp of $\Gamma_1(N)$ is equivalent to one with $c>0$. We choose $b$ and $d$ such that the matrix
$
M_{a/c}=\begin{psmallmatrix}
a & b \\ c & d
\end{psmallmatrix}$ is in $\SL_2(\Z)$ and $c'\mid d$. We want to apply \ref{prop:rescaledEta} to compute $\eta_\delta(M_{a/c}\tau)$ and note that we can choose $r=(c,\delta)$, $t=(c',\delta)$ and $s=0$. Therefore
\[
\eta_\delta(M_{a/c}\tau)=\lambda \sqrt{c\tau+d}\eta\left(\frac{(\delta,c)}{(\delta,c')}\tau\right)
\]
for some nonzero complex number $\lambda$. It follows that the smallest power of $q$ occuring in the Fourier expansion of
\[
\prod_{\delta\mid N}\eta(\delta\tau)^{r_\delta}|_{k}M_{a/c}
\] 
is
\[
\frac{1}{24}\sum_{\delta\mid N}\frac{(\delta,c)}{(\delta,c')} r_\delta,
\]
so that the eta quotient vanishes at the cusp $a/c$ if and only if the sum is positive. Since $(\delta,c)$ and $(\delta,c')$ only depend on $(c,N)$ it suffices to consider those $c$ that divide $N$.
\end{proof}

\section{Strongly reflective automorphic forms of singular weight}
Theorem \ref{thm1} states that there are only finitely many even lattices $L$ of squarefree level $N$ that split $\II_{1,1}\oplus\II_{1,1}(N)$ and carry a nonconstant reflective automorphic form for the discriminant kernel $\Gamma(L)$ and we have found bounds on the rank $n+2$ and the level $N$ of $L$ in Section 4. In this section we examine the remaining cases to find all strongly reflective automorphic forms of singular weight $(n-2)/2$ for $\Gamma(L)$. 

In the following $F$ will always be a strongly reflective modular form on an even lattice $L$ of squarefree level $N$ that splits $\II_{1,1}\oplus \II_{1,1}(N)$, as the study of strongly reflective automorphic forms for $\Gamma(L)$ can be reduced to the study of such $F$ by Theorem \ref{thm:converse} and Proposition \ref{prop:reflectiveProduct}.

Given $F$ and a divisor $d$ of $N$ we let
\[
M_d=\{\gamma\in D: F_\gamma=q^{-1/d}+O(1)\},
\]
and $c_d=|M_d|$. Note that the elements in $M_d$ all have order $d$ and norm $1/d\mod 1$ by Proposition \ref{prop:reflectiveProduct}. These numbers $c_d$ have to satisfy certain conditions, which we will now derive.
\begin{enumerate}[label=\Alph*.]
\item The numbers $c_d$ are obviously integers and satisfy $0\leq c_d\leq N(D_d,1)$. Moreover, since $-\Id=S^2\in\SL_2(\Z)$ acts in the Weil representation by sending $\e_\gamma\in D$ to $e(\sign(D)/4)\e_{-\gamma}$, the identity
\[
F_\gamma(\tau)=F_\gamma(-\Id\tau)=(-1)^{1-n/2}e(\sign(D)/4)F_{-\gamma}(\tau)=F_{-\gamma}(\tau)
\]
holds. Consequently all $c_d$ with $d>2$ are even integers.
\item We let 
\[
G=\frac{1}{|\Aut(D)|}\sum_{\sigma\in\Aut(D)}\sigma(F)
\]
be the symmetrization of $F$. Then $G$ is semi-reflective and symmetric. Let $\gamma\in D$ have order $d$ and norm $1/d\mod 1$. Then 
\[
G_\gamma=\tilde{c}_dq^{-1/d}+O(1)
\]
where $\tilde{c}_d=c_d/N(D_d,1)$. Since $G$ is symmetric, there exists a weakly holomorphic modular form $f$ of weight $1-n/2$ and character $\chi_D$ for $\Gamma_0(N)$ such that $G=F_{\Gamma_0(N),f,0}$ (see Corollary 5.5 in \cite{ScheithauerModForms}).

We have seen in the proof of Lemma \ref{lem:polesOfScalarForm} that $f|_{1-n/2}M_d=O(q^{-1/d'})$ for $d\mid N$. We let $a_d$ be the coefficient of $f|_{1-n/2}M_d$ at $q^{-1/d'}$. From the formula
\[
G_\gamma = \sum_{l\mid d'} \xi_l\frac{\sqrt{|D_l|}}{\sqrt{|D|}}l'g_{l',j_{\gamma,l'}}
\]
for the $\Gamma_0(N)$-lift (see Section 3), we find that
\[
\tilde{c}_d=\xi_{d'}\frac{\sqrt{|D_{d'}|}}{\sqrt{|D|}}da_{d'}
\]
or equivalently
\[
a_d=\xi_{d}^{-1}\frac{\sqrt{|D|}}{\sqrt{|D_{d}|}}\frac{c_{d'}}{d'\cdot N(D_{d'},1)}.
\]
Now let $g$ be a cusp form of weight $k=1+n/2$ for $\Gamma_0(N)$ with character $\chi_D$. Then $f\cdot g$ is a weakly holomorphic modular form of weight $2$ and trivial character for $\Gamma_0(N)$, so
\[
p=\sum_{M\in \Gamma_0(N)\backslash \SL_2(\Z)} (f\cdot g)|_2M
\]
is a weakly holomorphic modular form of weight $2$ for $\SL_2(\Z)$. Then $p$ can be identified with a meromorphic differential on the modular curve $X(1)$. Since $p$ is weakly holomorphic, this differential is holomorphic at all points except $\infty$. By the residue theorem, its residue at $\infty$ must vanish. But the residue at $\infty$ is (up to a nonzero constant) exactly the constant term in the Fourier expansion of the modular form $p$ at $\infty$, which is therefore $0$. Let $P$ be the set of cusps for $\Gamma_0(N)$. Then
\begin{align*}
p=\sum_{M\in \Gamma_0(N)\backslash \SL_2(\Z)} (f\cdot g)|_2M
=\sum_{s\in P}\sum_{\substack{M\in\Gamma_0(N)\backslash\SL_2(\Z)\\ M\infty =s}}(f\cdot g)|_2M.
\end{align*}
Note that the cusps of $\Gamma_0(N)$ are of the form $1/c$ for $c\mid N$ and that a set of representatives for the cosets of $\Gamma_0(N)$  in $\SL_2(\Z)$ sending $\infty$ to $1/c$ is given by $M_cT^j$ where $j=0,\ldots, t_c-1$ and $t_c=N/(N,c^2)$ (which is equal to $N/c=c'$ since $N$ is squarefree) is the width of $1/c$. Therefore
\begin{align*}
p&=\sum_{c\mid N} \sum_{j=0}^{t_c-1}(f\cdot g)|_2M_cT^j
\end{align*}
To obtain the constant Fourier coefficient of $p$ we must therefore add the constant Fourier coefficients of the functions $(f\cdot g)|_2M_cT^j$. But the constant coefficient of $(f\cdot g)|_2M_cT^j$ is equal to that of $(f\cdot g)|_2M_c$ for all $j$, so the constant coefficient of $p$ is given by
\begin{align*}
\sum_{c\mid N} t_c\cdot  [(f\cdot g)|_2M_c](0)&=\sum_{c\mid N} c' \sum_{\substack{\alpha\in \frac{1}{c'}\Z\\\alpha>0}}[f|_{1-n/2}M_c](-\alpha)[g|_{1+n/2}M_c](\alpha)\\
&=\sum_{c\mid N} c'\cdot  [f|_{1-n/2}M_c](-1/c')[g|_{1+n/2}M_c](1/c')
\end{align*}
where in the last step we have used that $f|_{1-n/2}M_c=O(q^{-1/c'})$. Letting $b_c$ be the coefficient of $g|_{1+n/2}M_c$ at $q^{1/c'}$, we have thus shown that
\[
0=\sum_{d\mid N} d'a_db_d=\sum_{d\mid N} \xi_{d}^{-1}\frac{\sqrt{|D|}}{\sqrt{|D_{d}|}}\frac{c_{d'}}{\cdot N(D_{d'},1)}b_d.
\]
We have described how to compute the coefficients $b_d$ for newforms and cusp forms that arise from newforms of lower levels in Section 5. One can always find a basis consisting of such cusp forms, so we can in fact compute all conditions coming from cusp forms. 

\item Another condition arises from a vector-valued Eisenstein series and has been described in Theorem 11.1 in \cite{ScheithauerClass}. In contrast to the other two conditions it uses that $[F_0](0)=n-2$, which follows from $\Psi$ having singular weight. In \cite{ScheithauerClass} $F$ is required to be symmetric, but this condition is not necessary if one replaces $c_dN_d$ by $c_d$ (note that these are two different $c_d$, the first is the one from \cite{ScheithauerClass}, which in our notation would be $\tilde{c}_d$, while the second one is as defined above). With this adjustment, Theorem 11.1 from \cite{ScheithauerClass} states that
\begin{equation}\label{eqn:Eisenstein}
\frac{k}{k-2}\frac{1}{B_{k,\psi}}\frac{L(k,\psi)}{L(k,\chi)}\frac{m^k}{N^k}\sum_{cd\mid N}\varepsilon_{c,d}c_d\frac{\sqrt{m_c|D_c|}}{\sqrt{m|D|}}\frac{N^k}{c^kd^{k-1}}=1
\end{equation}
where we use the following notation:
\allowdisplaybreaks
\begin{align*}
k&=1+n/2 \\
\chi&=\chi_D\\
m &\text{ is the conductor of } \chi \\
\psi &\text{ is the primitive character of modulus } m \text{ that induces } \chi \\
B_{k,\psi} &\text{ is the generalized Bernoulli number }\\
L(k,\cdot) &\text{ is the Dirichlet $L$-function of the character in the second}\\
&\text{ argument evaluated at } k\\
m_c&=(m,c)\\
\varepsilon_{c,d}&=\psi_c(N^2/(cdm_c))\psi_{c'}(-c)\frac{\psi_c(2)}{\psi(2)}\frac{\varepsilon_c}{\varepsilon}b_c\\
\psi_c,\psi_{c'} &\text{ are the unique characters of modulus } c \text{ and } c'\text{ such that } \psi=\psi_c\psi_{c'}\\
\varepsilon_c&=\prod_{p\mid c/m_c}\epsilon_p\legendre{-1}{p}^{n_p/2}\prod_{p\mid m_c}\epsilon_p\legendre{m_c/p}{p}\legendre{-1}{p}^{(n_p+1)/2}\\
\varepsilon&=\varepsilon_N\\
b_c&=\prod_{p\mid c/m_c}(-1)
\end{align*}
\end{enumerate}
The following proposition can be proved by computing the conditions A, B and C for the finitely many possible lattices from Table \ref{tab1}.
\begin{proposition}\label{prop:classificationSymmetric}
Let $L$ be an even lattice of squarefree level $N$ and signature $(n,2), n\geq 4$ such that $L$ splits $\II_{1,1}\oplus \II_{1,1}(N)$ and $F$ a strongly reflective modular form on $D=L'/L$ with $[F_0](0)=n-2$. If $F$ is symmetric, then it is of the form $F=F_{\Gamma_0(N),f,0}$ for one of the functions $f$ from Table \ref{tabIntro}.
\end{proposition}
\begin{proof}
We give one example, namely the case where $n=12$ and $N=3$. In this case $n_3$ must be odd by the formula for the signature of the Jordan components, so $\chi=\chi_D=\legendre{\cdot}{3}$. To compute the conditions B we observe that the space of cusp forms of weight $k=1+n/2=7$ and character $\chi$ for $\Gamma_0(3)$ has dimension 1 and is spanned by a newform $g$ with Fourier expansion
\[
g(\tau)=q-27q^3+64q^4+O(q^6).
\]  
There are two classes of cusps for $\Gamma_0(3)$, namely $1/c$ for $c=1,3$.  We let
\[
M_1=\begin{pmatrix}
1 & 2 \\ 1 & 3
\end{pmatrix} \text{ and } M_3=\begin{pmatrix}
1 & 0 \\ 3 & 1 
\end{pmatrix}.
\]
Then $M_3\in \Gamma_0(3)$ and
\[
g|_7M_3(\tau)=g(\tau)=q-27q^3+64q^4+O(q^6),
\]
so the coefficient $b_3=[g|_7M_3](1)=1$. Using Proposition \ref{prop:expansionNewforms} we find that
\begin{align*}
g|_7M_1(\tau)&=\chi_3(-1)\chi_1^{-1}(3)\cdot 3^{-7/2}g|_7W_3(\tau/3)\\
&=-3^{-7/2}g|_7W_3(\tau/3).
\end{align*}
By Proposition \ref{prop:Asai},
\[
g|_7W_3(\tau)=\lambda \sum_{n=1}^{\infty} a_n^{(3)}q^n
\]
with 
\begin{align*}
\lambda &= \chi_3(1)\chi_1(3)\chi_3(1)G(\chi_3)3^{-7/2}\overline{a_3}\\
&=\sqrt{3}\cdot i\cdot 3^{-7/2}\cdot (-27)\\
&=-i.
\end{align*}
It follows that
\[
g|_7M_1(\tau)=3^{-7/2}\cdot i \sum_{n=1}^\infty a_n^{(3)}q^{n/3}
\]
and $b_1=[g|_7M_1](1/3)=3^{-7/2}\cdot i \cdot a_1^{(3)}=3^{-7/2} i$. 

Therefore condition B states that 
\[
\xi_1^{-1}\frac{c_3\cdot \sqrt{|D|}}{N(D_3,1)}\cdot3^{-7/2}\cdot i+ \xi_3^{-1}c_1=0,
\]
where we have used that $N(D_1,1)=1$. Note that $\xi_3=1$ and that 
\[
\xi_1=e(\sign(D)/8)\legendre{-1}{|D|}=e((n-2)/8)\cdot (-1)=-e(1/4)=-i,\] so
\[
i\cdot \frac{c_3\cdot \sqrt{|D|}}{N(D_3,1)}\cdot 3^{-7/2}\cdot i+c_1=0,
\]
or equivalently
\[
c_3=\frac{3^{7/2}\cdot c_1\cdot N(D_3,1)}{\sqrt{|D|}}.
\]
To compute condition C we note that $\chi=\psi$, so $m=3$ and $B_{k,\psi}=98/3$. Then \eqref{eqn:Eisenstein} simplifies to
\[
\frac{7}{5}\cdot \frac{3}{98}\left(\varepsilon_{1,1}c_1\frac{1}{\sqrt{3|D|}}\cdot 3^7+\varepsilon_{1,3}c_3\frac{1}{\sqrt{3|D|}}\cdot 3+\varepsilon_{3,1}c_1\right)=1.
\]
To compute the signs $\varepsilon_{c,d}$ we note that $\legendre{-1}{3}^{(n_3+1)/2}=\epsilon_3$ by the signature formula for the Jordan components. A quick calculation then gives 
\[
\varepsilon_{1,1}=\varepsilon_{1,3}=\varepsilon_{3,1}=1,
\]
and hence
\[
\left(1+\frac{3^6\sqrt{3}}{\sqrt{|D|}}\right)c_1+\frac{\sqrt{3}}{\sqrt{|D|}}c_3=\frac{70}{3}.
\]
By condition A the number $c_1$ must be 0 or 1. If it is 0, then conditions B and C obviously have no common solution. Therefore $c_1=1$, so
\[
c_3=\frac{3^{7/2}\cdot N(D_3,1)}{\sqrt{|D|}}=3^{7/2}\left(\frac{\sqrt{|D|}}{3}+\frac{\sqrt{3}}{3}\right)
\]
by condition B. Inserting this into condition C gives
\[
28+756\frac{\sqrt{3}}{\sqrt{|D|}}=\frac{70}{3},
\]
which has no real solution for $|D|$. This completes the proof for this case. The other cases are similar.
\end{proof}
\renewcommand{\arraystretch}{1.2}
The result for the non-symmetric case is the following.
\begin{proposition}
Let $L$ be an even lattice of squarefree level $N$ and signature $(n,2), n\geq 4$ such that $L$ splits $\II_{1,1}\oplus \II_{1,1}(N)$ and $F$ a strongly reflective modular form on $D=L'/L$ with $[F_0](0)=n-2$. If $F$ is not symmetric, then $L$ and the numbers $c_d$ are one of the following.
\begin{center}
\begin{longtable}{c|c|c|c}
$n$ & $N$ & $L$ & \\ \hline
\multirow{4}{*}{10} & \multirow{4}{*}{2} & $\II_{10,2}(2_{\II}^{+4})$ & $c_1=0, c_2=2$ \\ \cline{3-4}
   &   & $\II_{10,2}(2_{\II}^{+6})$ & $c_1=0, c_2=4$ \\ \cline{3-4}
   &   & $\II_{10,2}(2_{\II}^{+8})$ & $c_1=0, c_2=8$ \\ \cline{3-4}
   &   & $\II_{10,2}(2_{\II}^{+10})$ & $c_1=0, c_2=16$ \\ \hline 
\multirow{2}{*}{8} & \multirow{2}{*}{3} & $\II_{8,2}(3^{+5})$ & $c_1=0, c_3=18 $ \\ \cline{3-4}
  &   & $\II_{8,2}(3^{-7})$ & $c_1=0, c_3=54 $ \\ \hline
\multirow{6}{*}{6} & 3 & $\II_{6,2}(3^{-4})$ & $c_1=0, c_3=4$ \\ \cline{2-4}
  & 5  & $\II_{6,2}(5^{+5})$ & $c_1=0, c_5=100$ \\ \cline{2-4}
  & \multirow{4}{*}{6} &$\II_{6,2}(2_{\II}^{-4}3^{+4})$ & $c_1=c_2=0,c_3=6,c_6=60$ \\ \cline{3-4}
& & $\II_{6,2}(2_{\II}^{-4}3^{-6})$ & $c_1=c_2=0,c_3=18,c_6=180$ \\ \cline{3-4}
& & $\II_{6,2}(2_{\II}^{+4}3^{-4})$ & $c_1=0,c_2=2,c_3=0,c_6=60$ \\ \cline{3-4}
& & $\II_{6,2}(2_{\II}^{+6}3^{-4})$ & $c_1=0,c_2=4,c_3=0,c_6=120$ \\ \hline
4 & 14  & $\II_{4,2}(2_{\II}^{+4}7^{-3})$ & $c_1=0,c_2=2,c_7=0,c_{14}=112$
\end{longtable}
\end{center}
\end{proposition}
\begin{proof}
If $N=1$ or $N$ is prime, then this is part of Theorem 6.27 in \cite{ScheithauerSingular}. In the other cases conditions A, B and C give the following list of possibilities:
\begin{center}
\begin{longtable}{c|c|c|c}
$n$ & $N$ & $L$ & \\ \hline
\multirow{4}{*}{8}
 & \multirow{4}{*}{6} &$\II_{8,2}(2_{\II}^{+4}3^{-3})$ &  $c_1=0,c_2=3,c_3=6,c_6=36$\\ \cline{3-4}
& &$\II_{8,2}(2_{\II}^{+6}3^{-3})$ & $c_1=0,c_2=7,c_3=6,c_6=84$\\ \cline{3-4}
&  &$\II_{8,2}(2_{\II}^{+4}3^{-7})$ &  $c_1=1,c_2=3,c_3=0,c_6=2268$ \\ \cline{3-4}
&  &$\II_{8,2}(2_{\II}^{+6}3^{-7})$ &  $c_1=1,c_2=7,c_3=0,c_6=5292$\\ \hline
\multirow{17}{*}{6} & \multirow{14}{*}{6} &\multirow{3}{*}{$\II_{6,2}(2_{\II}^{-4}3^{+4})$} & $c_1=c_2=0,c_3=6,c_6=60$ \\ \cline{4-4}
& & & $c_1=0,c_2=1,c_3=6,c_6=36$ \\ \cline{4-4}
& & & $c_1=0,c_2=2,c_3=6,c_6=12$ \\ \cline{3-4}
& & \multirow{2}{*}{$\II_{6,2}(2_{\II}^{-4}3^{-6})$} & $c_1=c_2=0,c_3=18,c_6=180$ \\ \cline{4-4}
& & & $c_1=0,c_2=2,c_3=24,c_6=84$\\ \cline{3-4}
& & \multirow{6}{*}{$\II_{6,2}(2_{\II}^{+4}3^{-4})$} & $c_1=0,c_2=2,c_3=0,c_6=60$ \\ \cline{4-4}
& & & $c_1=0,c_2=c_3=2,c_6=48$ \\ \cline{4-4}
& & & $c_1=0,c_2=2,c_3=4,c_6=36$ \\ \cline{4-4}  
& & & $c_1=0,c_2=2,c_3=6,c_6=24$ \\ \cline{4-4} 
& & & $c_1=0,c_2=2,c_3=8,c_6=12$ \\ \cline{4-4}  
& & & $c_1=0,c_2=2,c_3=10,c_6=0$ \\ \cline{3-4}
& & \multirow{3}{*}{$\II_{6,2}(2_{\II}^{+6}3^{-4})$} & $c_1=0,c_2=4,c_3=0,c_6=120$ \\ \cline{4-4} 
& & & $c_1=0,c_2=5,c_3=6,c_6=66$ \\ \cline{4-4}
& & & $c_1=0,c_2=6,c_3=12,c_6=12$\\ \cline{2-4}
& \multirow{3}{*}{10} & $\II_{6,2}(2_{\II}^{+4}5^{+3})$ & $c_1=0,c_2=3,c_5=20,c_{10}=90$\\ \cline{3-4}
& & \multirow{2}{*}{$\II_{6,2}(2_{\II}^{+4}5^{+5})$} &  $c_1=0,c_2=5,c_5=200,c_{10}=650$\\ \cline{4-4}
& &  &  $c_1=1,c_2=3,c_5=0,c_{10}=1950$\\  \hline
\multirow{10}{*}{4} & \multirow{4}{*}{6} & \multirow{4}{*}{$\II_{4,2}(2_{\II}^{-4}3^{-3})$} & $c_1=0,c_2=1,c_3=c_6=0$ \\ \cline{4-4}
& & & $c_1=c_2=0=c_3=0,c_6=12$ \\ \cline{4-4}
& & & $c_1=c_2=0,c_3=2,c_6=8$ \\ \cline{4-4}
& & & $c_1=c_2=0,c_3=c_6=4$ \\ \cline{2-4}
& 14 &  $\II_{4,2}(2_{\II}^{+4}7^{-3})$ & $c_1=0,c_2=2,c_7=0,c_{14}=112$\\ \cline{2-4}
& 15 & $\II_{4,2}(3^{+3}5^{-3})$ & $c_1=0,c_3=4,c_5=10,c_{15}=0$ \\ \cline{2-4}  
& \multirow{4}{*}{30} & $\II_{4,2}(2_{\II}^{+4}3^{-3}5^{+3})$ &  $c_1=0,c_2=3,c_3=6,c_5=20,c_6=36,c_{10}=90,c_{15}=0,c_{30}=360$ \\ \cline{3-4}
& & \multirow{3}{*}{$\II_{4,2}(2_{\II}^{+4}3^{+3}5^{-3})$}  &  $c_1=c_2=1,c_3=8,c_5=20,c_6=30,c_{10}=100,c_{15}=120,c_{30}=360$\\ \cline{4-4}
&  & &  $c_1=1,c_2=2,c_3=4,c_5=10,c_6=24,c_{10}=80,c_{15}=120,c_{30}=720$\\ \cline{4-4}
& & & $c_1=1,c_2=3,c_3=c_5=0,c_6=18,c_{10}=60,c_{15}=120,c_{30}=1080$ 
\end{longtable}
\end{center}
For these remaining cases we proceed as follows. Let $\gamma\in D$. Then $F_\gamma$ is a weakly holomorphic modular form of weight $1-n/2$ and character $\chi_\gamma$ for $\Gamma_1(N)$. If $h$ is a cusp form of weight $k=1+n/2$ and character $\overline{\chi}_\gamma$ for $\Gamma_1(N)$, then
\[
p=\sum_{M\in \Gamma_1(N)\backslash \SL_2(\Z)}(F_\gamma|_{1-n/2}M)(h|_{1+n/2}M)
\]
is a weakly holomorphic modular form of weight $2$ for $\SL_2(\Z)$, so the constant term in its Fourier expansion at $\infty$ must vanish. We call $p$ the pairing of $F_\gamma$ with $h$. Note that the Fourier expansion of $F_\gamma|_{1-n/2}M$ can be computed using \eqref{eqn:ComponentsTransformation}. For $h$ we will choose eta quotients, so that $h|_{1+n/2}M$ can be computed using Propositions \ref{prop:etaTransformation} and \ref{prop:rescaledEta}. We can therefore write the constant coefficient $[p](0)$ of $p$ in terms of the Fourier coefficients of $F$ at negative powers of $q$. Since $[p](0)=0$, this gives conditions on these coefficients. To simplify the calculations we note that 
\[
p=\sum_{s\in P}\sum_{\substack{M\in\Gamma_1(N)\backslash\SL_2(\Z) \\ M\infty=s}}(F_\gamma|_{1-n/2}M)(h|_{1+n/2}M)
\]
where $P$ is the set of cusps for $\Gamma_1(N)$ and that a set of representatives for the cosets of $\Gamma_1(N)$ in $\SL_2(\Z)$ mapping $\infty$ to a cusp $a/c$ is given by $M_{a/c}T^j$ where $M_{a/c}$ is any matrix in $\SL_2(\Z)$ mapping $\infty$ to $a/c$ and $j=0,\ldots,t_c-1$ where $t_c=N/(c,N)$ is the width of $a/c$. Note that we can suppose that $M_{a/c}=\begin{psmallmatrix} a & b \\ c &d \end{psmallmatrix}$ satisfies $d=0\mod t_c$. Therefore
\[
p=\sum_{s\in P}t_c\cdot (F_\gamma|_{1-n/2}M_{a/c})(h|_{1+n/2}M_{a/c}),
\]
so we only need to compute one Fourier expansion of $F_\gamma$ and $h$ for each cusp of $\Gamma_1(N)$.

We give a detailed example of this for $n=8$ and $N=6$, while we only sketch the proof for the remaining cases, as they are very similar. If $\gamma\in D$, then we let $a_\gamma^d=|M_d\cap \gamma^\perp|$ and
\[
\delta_{\gamma,M_d}=\begin{cases}1 &\text{ if } \gamma\in M_d, \\
0&\text{ otherwise.}\end{cases}
\]
If $\gamma$ has order $d$, then we let $N_\gamma^{d,e}$ be the number of elements in $M_e$ that project onto $\gamma$ under the natural projection $D=D_d\oplus D^d\to D_d$.

\begin{itemize}
\item
If $n=8$ and $N=6$ then we let $\beta\in D$ be an element of order $2$ and norm $1/2\mod 1$ and let $h_1(\tau)=\eta(\tau)^6\eta(2\tau)^3\eta(3\tau)^2\eta(6\tau)^{-1}$. Then $h_1$ is a cusp form of weight $1+n/2=5$ and character $\overline{\chi}_\beta=\chi_\beta$ by Propositions \ref{prop:DHS} and \ref{prop:etaCuspForm}. The group $\Gamma_1(6)$ has four classes of cusps, namely $s=1/6,1/3,1/2$ and $1/1$. We choose the matrices $M_{a/c}$ as follows:
\[
M_{1/6}=\begin{pmatrix}
1 & 0 \\ 6 & 1 
\end{pmatrix} \quad M_{1/3}=\begin{pmatrix}
1 & 1 \\ 3 & 4
\end{pmatrix} \quad M_{1/2}=\begin{pmatrix}
1 & 1  \\ 2 & 3
\end{pmatrix} \quad M_{1/2}=\begin{pmatrix}
1 & 5 \\ 1 & 6
\end{pmatrix}
\]
Then \eqref{eqn:ComponentsTransformation} gives
\allowdisplaybreaks
\begin{align*}
F_\beta|_{-3}M_{1/6}&=F_\beta = \delta_{\beta,M_2}q^{-1/2}+O(q^{1/2}),\\
F_\beta|_{-3}M_{1/3}&=\xi(M_{1/3})\frac{\sqrt{|D_3|}}{\sqrt{|D|}}\sum_{\mu\in D^3}e(-\beta\mu)e(-\beta^2/2)F_{\beta+\mu}\\
&=-\epsilon_2\frac{1}{\sqrt{|D_2|}}\sum_{\mu\in D_2}e(\beta\mu)F_\mu\\
&=-\epsilon_2\frac{1}{\sqrt{|D_2|}}\left( c_1q^{-1}+\left(a_\beta^2-(c_2-a_\beta^2)\right)q^{-1/2}\right)+O(1)\\
&=-\epsilon_2\frac{1}{\sqrt{|D_2|}}\left(c_1q^{-1}+\left(2a_\beta^2-c_2\right)q^{-1/2}\right)+O(1), \\
F_\beta|_{-3}M_{1/2}&=\xi(M_{1/2})\frac{\sqrt{|D_2|}}{\sqrt{|D|}}\sum_{\mu\in D^2}e(-\beta\mu)e(-\beta^2/2)F_{\beta+\mu}\\
&=-i\epsilon_2\frac{1}{\sqrt{|D_3|}}\sum_{\mu\in \beta+D_3}F_{\mu}\\
&=-i\epsilon_2\frac{1}{\sqrt{|D_3|}}\left(\delta_{\beta,M_2}q^{-1/2}+N_\beta^{2,6} q^{-1/6}\right)+O(q^{1/6})\\
F_\beta|_{-3}M_{1/1}&=\xi(M_{1/1})\frac{\sqrt{|D_1|}}{\sqrt{|D|}}\sum_{\mu\in D^1}e(-\beta\mu)e(-\beta^2/2)F_{\beta+\mu}\\
&=i\frac{1}{\sqrt{|D|}}\sum_{\mu\in D}e(\beta\mu)F_\mu\\
&=i\frac{1}{\sqrt{|D|}}\left(c_1q^{-1}+(2a_\beta^2-c_2)q^{1/2}+c_3q^{-1/3}+(2a_\beta^6-c_6)q^{-1/6}\right)+O(1).
\end{align*}
The Fourier expansions of $h_1$ can be computed using Propositions \ref{prop:etaTransformation} and \ref{prop:rescaledEta} and are given by
\begin{align*}
h_1|_5M_{1/6}&=h_1=\eta(\tau)^6\eta(2\tau)^3\eta(3\tau)^2\eta(6\tau)^{-1}=q^{1/2}+O(q^{3/2}),\\
h_1|_5M_{1/3}&=-\frac{1}{2}\eta(\tau)^6\eta(\tau/2)^3\eta(3\tau)^2\eta(3\tau/2)^{-1}=-\frac{1}{2}q^{1/2}+\frac{3}{2}q+O(q^{3/2}),\\
h_1|_5M_{1/2}&=-i\frac{\sqrt{3}}{3}\eta(\tau)^6\eta(2\tau)^3\eta(\tau/3)^2\eta(2\tau/3)^{-1}=-i\frac{\sqrt{3}}{3}\left(q^{1/2}-2q^{5/6}\right)+O(q^{3/2}),\\
h_1|_5M_{1/1}&=i\frac{\sqrt{3}}{6}\eta(\tau)^6\eta(\tau/2)^3\eta(\tau/3)^2\eta(\tau/6)^{-1}=i\frac{\sqrt{3}}{6}\left(q^{1/3}+q^{1/2}-2q^{5/6}-3q\right)+O(q^{4/3}).
\end{align*}
Therefore the constant coefficient of the pairing of $F_\beta$ with $h_1$ is given by
\begin{align*}
0=\delta_{\beta,M_2}+\epsilon_2\frac{1}{\sqrt{|D_2|}}\left(-3c_1+2a_\beta^2-c_2\right)-\epsilon_2\frac{\sqrt{3}}{\sqrt{|D_3|}}\delta_{\beta,M_2}-\frac{\sqrt{3}}{\sqrt{|D|}}\left(-3c_1+2a_\beta^2-c_2+c_3\right)
\end{align*}
which is equivalent to
\begin{equation}\label{eqn:eightSixA}
a_\beta^2=\frac{3}{2}c_1+\frac{1}{2}c_2-\epsilon_2\frac{\sqrt{|D_2|}}{2}\delta_{\beta,M_2}-\frac{1}{2}\left(1-\epsilon_2\frac{\sqrt{|D_3|}}{\sqrt{3}}\right)^{-1}c_3.
\end{equation}
If we choose $\beta\in M_2$, then we obtain $a_\beta^2=1$ for all remaining lattices. The function $h_2(\tau)=\eta(\tau)^7\eta(2\tau)^{-2}\eta(3\tau)^7\eta(6\tau)^{-2}$ is also a cusp form of weight $5$ and character $\chi_\beta$ for $\Gamma_1(6)$. The pairing of $F_\beta$ with $h_2$ can be computed similarly and one obtains the following condition.

\begin{align*}
0=&\delta_{\beta,M_2}-\epsilon_2\frac{8}{\sqrt{|D_2|}}c_1+\epsilon_2\frac{1}{3\sqrt{3}\sqrt{|D_3|}}N_\beta^{2,6}-\epsilon_2\frac{7}{3\sqrt{3}\sqrt{|D_3|}}\delta_{\beta,M_2}\\&-\frac{16}{3\sqrt{3}\sqrt{|D|}}\left(2a_\beta^2-c_2\right)
-\frac{8}{\sqrt{3}\sqrt{|D|}}c_1-\frac{8}{3\sqrt{3}\sqrt{|D|}}c_3,
\end{align*}
which gives
\begin{align*}
N_\beta^{2,6} = &\left(7-\epsilon_23\sqrt{3}\sqrt{|D_3|}\right)\delta_{\beta,M_2}+\frac{24}{\sqrt{|D_2|}}\left(\epsilon_2+\sqrt{3}\sqrt{|D_3|}\right)c_1\\
&+\epsilon_2\frac{8}{\sqrt{|D_2|}}c_3+\epsilon_2\frac{16}{\sqrt{|D_2|}}\left(2a_\beta^2-c_2\right).
\end{align*}
Letting $\beta\in M_2$ and inserting $a_\beta^2=1$ and the other values gives a negative value for $N_\beta^{2,6}$ in the cases where $n_3=3$, which is of course impossible. For $L_1=\II_{8,2}(2_{\II}^{+4}3^{-7})$ we obtain $N_\beta^{2,6} = 252$ and for $L_2=\II_{8,2}(2_{\II}^{+6}3^{-7})$ we obtain $N_\beta^{2,6}=0$. Next we let $\alpha\in D$ be an element of order $6$ and norm $1/6\mod 1$ and let $h_3(\tau)=\eta(\tau)^5\eta(2\tau)^4\eta(3\tau)^5\eta(6\tau)^{-4}$. Then $h_3$ is a cusp form of weight $5$ and character $\overline{\chi}_\alpha$ for $\Gamma_1(6)$ and we can compute the pairing of $F_\alpha$ with $h_3$ as above and obtain
\begin{equation}\label{eqn:eightSixB}
0=\delta_{\alpha,M_6}+\epsilon_2\frac{\sqrt{3}}{\sqrt{|D_3|}}\delta_{3\alpha,M_2}-\frac{48\sqrt{3}}{\sqrt{|D|}}c_1-\frac{8\sqrt{3}}{\sqrt{|D|}}(2a_\alpha^2-c_2)-\frac{2\sqrt{3}}{\sqrt{|D|}}\left(\frac{3}{2}a_\alpha^3-c_3\right).
\end{equation}
For both $L_1$ and $L_2$ there exists an $\alpha\in M_6$ with $3\alpha\notin M_2$, because
\[
\sum_{\beta\in M_2}N_\beta^{2,6}<c_6.
\]
For such an $\alpha$ we can compute $a_\alpha^2$ using \eqref{eqn:eightSixA} (note that $a_\alpha^2=a_{3\alpha}^2$) and obtain
\[
a_\alpha^2=\frac{3+c_2}{2}.
\]
Inserting this and the other values into \eqref{eqn:eightSixB} shows that
\[
0=1-\frac{72\sqrt{3}}{\sqrt{|D|}},
\]
i.e. $|D|=2^6\cdot 3^5$, which is not the case. This completes the proof for $n=8$ and $N=6$.
\item 
If $n=6$ and $N=6$, then we can compute the pairing of $F_\beta$ for $\beta\in D$ of order $2$ and norm $1/2\mod 1$ with $h_1(\tau)=\eta(\tau)^9\eta(2\tau)^{-3}\eta(3\tau)\eta(6\tau)$ and obtain 
\begin{align}\label{eqn:sixSixA}
0&=\delta_{\beta,M_2}+\epsilon_2\frac{12}{\sqrt{|D_2|}}c_1+\epsilon_2\frac{4}{\sqrt{|D_2|}}\left(2a_\beta^2-c_2\right)-\epsilon_2\frac{1}{\sqrt{|D_3|}}N_\beta^{2,6}+\epsilon_2\frac{1}{\sqrt{|D_3|}}\delta_{\beta,M_2} \nonumber \\
&+\frac{12}{\sqrt{|D|}}c_1-\frac{4}{\sqrt{|D|}}\left(2a_\beta^2-c_2\right)+\frac{4}{\sqrt{|D|}}c_3 \nonumber \\
=&\left(1+\epsilon_2\frac{1}{\sqrt{|D_3|}}\right)\left(\delta_{\beta,M_2}+\epsilon_2\frac{12}{\sqrt{|D_2|}}c_1+\epsilon_2\frac{4}{\sqrt{|D_2|}}\left(2a_\beta^2-c_2\right)\right)\\
&-\epsilon_2\frac{1}{\sqrt{|D_3|}}N_\beta^{2,6}-\frac{4}{\sqrt{|D|}}c_3. \nonumber 
\end{align}
We describe how this can be used to show that the case $L=\II_{6,2}(2_{\II}^{-4}3^{+4})$ with $c_2=1$ cannot occur. Let $\beta\notin M_2$. Then $a_\beta^2$ is 0 or 1. If it is 0, then $N_\beta^{2,6}=-2$, which is impossible, so $a_\beta^2=1$. In this case $N_\beta^{2,6}=14$. There are $N(D_2,1)-c_2=9$ such elements $\beta$. Therefore $c_6\geq 9\cdot 14$, which is a contradiction. With a similar argument one can eliminate the other cases, except those stated in the proposition and the case $L=\II_{6,2}(2_{\II}^{+6}3^{-4})$ with $c_2=5$. In this case we let $\beta\in M_2$. Then \eqref{eqn:sixSixA}
 gives
 \[
 N_\beta^{2,6} = 10 a_\beta^2-18.
 \]
 We then pair $F_\beta$ with $h_2=\eta(\tau)^4\eta(2\tau)^{-2}\eta(6\tau)^6$ which gives
 \begin{align}\label{eqn:sixSixB}
0=&\epsilon_2\frac{1}{\sqrt{|D_2|}}c_1+\epsilon_2\frac{1}{2\sqrt{|D_2|}}\left(2a_\beta^2-c_2\right)-\epsilon_2\frac{1}{9\sqrt{|D_3|}}N_\beta^{2,6}+\frac{1}{18\sqrt{|D|}}\left(2a_\beta^6-c_6\right) \nonumber \\
&+\frac{1}{\sqrt{|D|}}c_1+\frac{1}{2\sqrt{|D|}}\left(2a_\beta^2-c_2\right)-\frac{1}{3\sqrt{|D|}}c_3 \nonumber \\
=&\left(\epsilon_2+\frac{1}{\sqrt{|D_3|}}\right)\left(\frac{1}{\sqrt{|D_2|}}c_1+\frac{1}{2\sqrt{|D_2|}}\left(2a_\beta^2-c_2\right)\right)-\epsilon_2\frac{1}{9\sqrt{|D_3|}}N_\beta^{2,6}\\
&+\frac{1}{18\sqrt{|D|}}\left(2a_\beta^6-c_6\right)-\frac{1}{3\sqrt{|D|}}c_3, \nonumber
\end{align}
In our case this simplifies to
\[
180a_\beta^2-16N_\beta^{2,6}+2a_\beta^6=552.
\]
Using that $N_\beta^{2,6} = 10a_\beta^2-18$ this gives
\[
10a_\beta^2+a_\beta^6=132,
\]
which is impossible because $a_\beta^2\leq c_2=5$ and $a_\beta^6\leq c_6=66$.
\item  For $n=6$ and $N=10$ we let $\beta\in D$ have order $2$ and norm $1/2\mod 1$ and we compute the pairing of $F_\beta$ with $h(\tau)=\eta(\tau)^7\eta(5\tau)$. This gives
\[
0=2\delta_{\beta,M_2}-\epsilon_2\frac{4}{\sqrt{|D_2|}}(2a_\beta^2-c_2)-\epsilon_2\frac{2\sqrt{5}}{\sqrt{|D_5|}}\delta_{\beta,M_2}+\frac{4\sqrt{5}}{\sqrt{|D|}}(2a_\beta^2-c_2)
\]
which gives
\[
a_\beta^2=\epsilon_2\frac{\sqrt{|D_2|}}{4}\delta_{\beta,M_2}+\frac{1}{2}c_2.
\]
In the cases we are considering $c_2\neq N(D_2,1)$, so we can choose $\beta\notin M_2$. Then $a_\beta^2=c_2/2$, so $c_2$ must be even, giving a contradiction.
\item For $n=4$ and $N=6$ the only possible lattice is $L=\II_{4,2}(2_{\II}^{-4}3^{-3})$. We let $\beta\in D$ have order $2$ and norm $1/2\mod 1$ and pair $F_\beta$ with $h_1(\tau)=\eta(\tau)^3\eta(3\tau)^3$ to obtain
\begin{equation}\label{eqn:fourSixA}
0=\frac{4}{3}\delta_{\beta,M_2}+\frac{4}{3}a_\beta^2-\frac{2}{3}c_2-\frac{1}{9}N_\beta^{2,6}-\frac{1}{9}a_\beta^6+\frac{1}{18}c_6.
\end{equation}
If $c_2=1$, then $c_6=0$ and therefore $a_\beta^6$ and $N_\beta^{2,6}$ are also $0$. Then
\[
a_\beta^2=\frac{1}{2}-\delta_{\beta,M_2}\notin\Z
\]
which is not possible. Therefore $c_2=\delta_{\beta,M_2}=a_\beta^2=0$.
\begin{itemize}
\item If $c_3=c_6=4$, then
\[N_\beta^{2,6}=\frac{1}{2}c_6-a_\beta^6=2-a_\beta^6,
\]
by \eqref{eqn:fourSixA}. Then $N_\beta^{2,6}\leq 1$ because $N_\beta^{2,6}\leq a_\beta^6$. But $N_\beta^{2,6}$ must be even because if $\alpha\in M_6$ projects onto $\beta$, then so does $-\alpha$, which must also be in $M_6$ because $F_\alpha=F_{-\alpha}$. Therefore $N_\beta^{2,6}=0$. Since this holds for any $\beta\in D$ of order $2$ and norm $1/2\mod 1$, it follows that $c_6=0$, which is a contradiction.
\item If $c_3=2$ and $c_6=8$, then we let $\gamma\in M_3$. Then $M_3=\{\pm\gamma\}$, so $\gamma^\perp\cap M_3=\emptyset$, i.e. $a_\gamma^3=0$. Computing the pairing of $F_\gamma$ with $h_2(\tau)=\eta(\tau)^2\eta(2\tau)^3\eta(3\tau)^2\eta(6\tau)^{-1}$ yields
\begin{align*}
0&=\frac{5}{4}\delta_{\gamma,M_3}-\frac{5}{8}a_\gamma^3+\frac{5}{24}c_3-\frac{1}{8}a_\gamma^6+\frac{1}{24}c_6\\
&=2-\frac{1}{8}a_\gamma^6,
\end{align*}
i.e. $a_\gamma^6 = 16$, which is impossible because $a_\gamma^6$ cannot be larger than $c_6$.
\item If $c_3=0$ and $c_6=12$, then \eqref{eqn:fourSixA} gives
\[
N_\beta^{2,6}=\frac{1}{2}c_6-a_\beta^6 = 6-a_\beta^6.
\]
As $N_\beta^{2,6}\leq a_\beta^6$ and even, it must be $0$ or $2$. It follows that there are six elements $\beta_1,\ldots,\beta_6\in D$ of order $2$ and norm $1/2\mod 1$ with $N_{\beta_i}^{2,6}=2$ and $a_{\beta_i}^6 = 4$. Hence, each of the $\beta_i$ is orthogonal to itself and exactly one of the others. Next we let $\alpha\in D$ have order $6$ and norm $1/6\mod 1$. Computing the pairing of $F_\alpha$ with $h_3(\tau)=\eta(\tau)^{10}\eta(2\tau)^{-3}\eta(3\tau)^{-2}\eta(6\tau)$ gives
\[
0=\delta_{\alpha,M_6}+\frac{3}{2}N_{3\alpha}^{\perp}-1-2M_{4\alpha}^\perp+N_{4\alpha}^{3,6}
\]
where $N_{3\alpha}^\perp$ is the number of elements in $M_6\cap (4\alpha)^\perp$ that project onto $3\alpha$ under the natural projection $D=D_2\oplus D_3\to D_2$ and $M_{4\alpha}^\perp$ is the number of elements in $M_6\cap (3\alpha)^\perp$ that project onto $4\alpha$ under $D=D_2\oplus D_3\to D_3$.

Suppose there is an element $\delta\in D$ of order $3$ and norm $2/3\mod 1$ such that $\delta+\beta\notin M_6$ for all $\beta\in D$ of order $2$ and norm $1/2\mod 1$. Then we let $\alpha_1=\delta+\beta_1$. Then $4\alpha_1=\delta$, so both $\delta_{\alpha_1,M_6}$ and $N_{4\alpha_1}^{3,6}$ are $0$. Then of course $M_{4\alpha_1}^\perp$ must also be $0$, so $N_{3\alpha_1}^\perp=2/3$, which is not an integer and therefore not possible.

There is therefore no such $\delta$. Since there are $N(D_3,2)=12=c_6$ elements of order $3$ and norm $2/3\mod 1$ in $D$, it follows that for every element $\delta$ of order $3$ and norm $2/3\mod 1$ there is exactly one $\beta\in D$ of order $2$ and norm $1/2\mod 1$ with $\beta+\delta\in M_6$. Now let $\alpha\in M_6$. Then $N_{4\alpha}^{3,6}=1$ as just explained, i.e. there is exactly one element in $M_6$ which projects onto $4\alpha$ under $D\to D_3$. This element is of course $\alpha$, which is  orthogonal to $3\alpha$, so $M_{4\alpha}^\perp=1$. Then again $N_{3\alpha}^\perp=2/3$, again giving a contradiction. This proves that this case cannot occur.
\end{itemize}
\item If $n=4$ and $N=15$, then $L=\II_{4,2}(3^{+3}5^{-5})$ and $c_3=4$, so there exists a $\gamma\in M_3$. Pairing $F_\gamma$ with $h(\tau)=\eta(3\tau)\eta(5\tau)^7\eta(15\tau)^{-2}$ gives
\[
0=\frac{376}{125}-\frac{36}{125}a_\gamma^3,
\]
which is a contradiction, because $a_\gamma^3$ must be an integer.
\item For $n=4$ and $N=30$ we let $\beta\in D$ have order $2$ and norm $1/2\mod 1$. Pairing $F_\beta$ with 
\[h(\tau)=\eta(\tau)^3\eta(2\tau)^{-1}\eta(5\tau)^2\eta(6\tau)\eta(10\tau)
\eta(15\tau)\eta(30\tau)^{-1}
\]
gives
\begin{align*}
0=&\left(4-\epsilon_5\frac{8\sqrt{5}}{5\sqrt{|D_5|}}+\frac{4\sqrt{15}}{5\sqrt{|D_{15}|}}\right)\delta_{\beta,M_2}-\left(\frac{8}{\sqrt{|D_2|}}-\epsilon_5\frac{8\sqrt{3}}{\sqrt{|D_6|}}+\epsilon_5\frac{8\sqrt{5}}{\sqrt{|D_{10}|}}+\frac{24\sqrt{15}}{5\sqrt{|D|}}\right)c_1\\
&+\frac{16\sqrt{15}}{5\sqrt{|D|}}c_3-\frac{4\sqrt{15}}{5\sqrt{|D_{15}|}}N_\beta^{2,6}+\frac{4\sqrt{15}}{5\sqrt{|D_{15}|}}N_\beta^{2,10}-\left(\epsilon_5\frac{8\sqrt{5}}{5\sqrt{|D_{10}|}}+\frac{8\sqrt{15}}{5\sqrt{|D|}}\right)c_5\\
&+\left(\frac{8\sqrt{15}}{5\sqrt{|D|}}-\epsilon_5\frac{24\sqrt{5}}{5\sqrt{|D_{10}|}}\right)\left(2a_\beta^2-c_2\right)-\frac{8\sqrt{15}}{5\sqrt{|D|}}\left(2a_\beta^6-c_6\right)
\end{align*}
Suppose $L=\II_{4,2}(2_{\II}^{+4}3^{-3}5^{+3})$. Since $c_2=3\neq 0$, we can let $\beta\in M_2$. Then the condition simplifies to
\[
a_\beta^2=\frac{33}{4}+\frac{1}{8}N_\beta^{2,10}-\frac{1}{8}N_\beta^{2,6}-\frac{1}{8}a_\beta^6.
\]
Note that $N_\beta^{2,6}$ is at most equal to the number of elements of order $3$ and norm $2/3\mod 1$ in $D$, i.e. $N_\beta^{2,6}\leq N(D_3,2)=12$ and $a_\beta^6\leq c_6=36$. Therefore $a_\beta^2\geq 3$. Since $a_\beta^2\leq c_2=3$, it follows that $a_\beta^2=3$. Since this holds for any $\beta\in M_2$, the three elements in $M_2$ are pairwise orthogonal. However, one easily checks that it is not possible to choose three pairwise orthogonal elements of order $2$ and norm $1/2\mod 1$ in $2_{\II}^{+4}$. Therefore, this case cannot occur.
A similar argument also works in the other remaining cases.
\end{itemize}
\end{proof}

\begin{proposition}\label{prop:nonSymmetricIsLift}
Let $L$ be an even lattice of squarefree level $N$ and signature $(n,2), n\geq 4$ such that $L$ splits $\II_{1,1}\oplus \II_{1,1}(N)$ and $F$ a strongly reflective modular form on $D=L'/L$ with $[F_0](0)=n-2$. If $F$ is not symmetric, then there exists an isotropic subgroup $H\subset D$ such that $F$ is the lift of a symmetric strongly reflective modular form $F_H$ for $D_H=H^\perp/H$ with $[(F_H)_0](0)=n-2$ on $H$.
\end{proposition}
\begin{proof}
We already know that $L$ and the numbers $c_d$ are one of those from the previous proposition. For prime level the claim follows from Theorem 6.27 in \cite{ScheithauerSingular}, so we can assume that the level $N$ is 6 or 14. 
\begin{itemize}
\item Suppose the level is $6$. If $c_2>0$, then $L=\II_{6,2}(2_{\II}^{+n_2}3^{-4})$ for $n_2=4$ or $n_2=6$. Let $\beta\in M_2$. Using that $c_1=c_3=0$ and $c_2=2^{n_2/2-1}$, \eqref{eqn:sixSixA} gives 
\[
N_\beta^{2,6} = \frac{80}{2^{n_2/2}}a_\beta^2- 10.
\]
We then pair $F_\beta$ with $h_2(\tau)=\eta(\tau)^4\eta(2\tau)^{-2}\eta(6\tau)^6$. The resulting condition was already stated in \eqref{eqn:sixSixB}. Using that $c_6=30\cdot 2^{n_2/2-1}$, this gives
\[
N_\beta^{2,6}=\frac{90}{2^{n_2/2}}a_\beta^2+\frac{1}{2^{n_2/2}}a_\beta^6-30.
\] 
These two conditions together yield
\[
N_\beta^{2,6}=150-\frac{8}{2^{n_2/2}}a_\beta^6\geq 30
\]
because $a_\beta^6\leq c_6=30\cdot 2^{n_2/2-1}$. As there are only $N(D_3,2)=30$ elements of order $3$ and norm $2/3\mod 1$, it follows that $N_\beta^{2,6} = 30$ and therefore $a_\beta^2=c_2$ and $a_\beta^6=c_6$, i.e. the elements in $M_2$ are pairwise orthogonal and 
\[
M_6=M_2+A_{2/3},
\]
where $A_{2/3}$ is the set of elements of order $3$ and norm $2/3\mod 1$ in $D$. Next, we let $H$ be the set of isotropic elements in $D_2$ that are orthogonal to $M_2$. Let $\beta\in M_2$ and $\mu\in H$. Then $\beta+\mu$ obviously has order $2$ and norm $1/2\mod 1$, but it is also in $M_2$: Suppose it were not. Then \eqref{eqn:sixSixA} gives $N_{\beta+\mu}^{2,6}=20$, which is impossible because $M_6=M_2+A_{2/3}$. So $\beta+\mu\in M_2$ and $\beta+H\subset M_2$. But the other inclusion $M_2\subset \beta+H$ is also true: Let $\beta'\in M_2$. Then $\beta'-\beta$ is isotropic (because $\beta'$ and $\beta$ are orthogonal) and in $M_2^\perp$ (because $\beta'$ and $\beta$ are). Hence $M_2=\beta+H$. We show that $H$ is a group: Let $\mu_1,\mu_2\in H$ and write $\mu_1=\beta_1-\beta$ and $\mu_2=\beta_2-\beta$ with $\beta_1,\beta_2\in H$. Then $\mu_1-\mu_2=\beta_1-\beta_2$ is isotropic and orthogonal to $M_2$ and hence in $H$. On $D_H=H^\perp/H=2^{+2}3^{-4}$ there is a symmetric strongly reflective modular form $F_H$ with $[(F_H)_0](0)=4$, which can be lifted on $H$. Then the principal parts of the lift of $F_H$ and $F$ coincide, so $F$ must be the lift of $F_H$.

If $c_3>0$, then $L=\II_{6,2}(2_{\II}^{-4}3^{\epsilon_3n_3})$ where $\epsilon_3n_3=+4$ or $-6$, $c_3=2\cdot 3^{n_3/2-1}$ and $c_6=20\cdot 3^{n_3/2-1}$. Let $\gamma\in M_3$. Pairing $F_\gamma$ with $\eta(\tau)^8$ shows that
\begin{equation}\label{eqn:sixSixC}
0=\delta_{\gamma,M_3}+\epsilon_2\frac{2}{\sqrt{|D_2|}}\delta_{\gamma,M_3}-\epsilon_2\frac{3}{\sqrt{|D_3|}}\left(\frac{3}{2}a_\gamma^3-\frac{1}{2}c_3\right)-\frac{6}{\sqrt{|D|}}\left(\frac{3}{2}a_\gamma^3-\frac{1}{2}c_3\right),
\end{equation}
which in this case yields $a_\gamma^3=0$, i.e. no element in $M_3$ is orthogonal to $\gamma$. Next we pair $\gamma$ with $\eta(\tau)^{-1}\eta(2\tau)^9\eta(3\tau)^3\eta(6\tau)^{-3}$. Then
\begin{align*}
0=&\delta_{\gamma,M_3}-\epsilon_2\frac{1}{4\sqrt{|D_2|}}\delta_{\gamma,M_3}-\epsilon_2\frac{9}{\sqrt{|D_3|}}c_1+\frac{9}{\sqrt{|D|}}c_1-\frac{9}{2\sqrt{|D|}}c_2\\
&-\frac{9}{4\sqrt{|D|}}\left(\frac{3}{2}a_\gamma^3-\frac{1}{2}c_3\right)-\frac{3}{4\sqrt{|D|}}\left(\frac{3}{2}a_\gamma^6-\frac{1}{2}c_6\right), 
\end{align*}
which shows that in our case $a_\gamma^6=c_6$, i.e. $M_3\perp M_6$. Next let $\delta\in D$ have order $3$ and norm $2/3\mod 1$. Pairing $F_\delta$ with $\eta(\tau)^9\eta(2\tau)^{-1}\eta(3\tau)^{-3}\eta(6\tau)^3$ shows that
\begin{equation}\label{eqn:sixSixD}
0=\left(\epsilon_2-\frac{1}{\sqrt{|D_2|}}\right)\left(\frac{18}{\sqrt{|D_3|}}c_1+\frac{3}{\sqrt{|D_3|}}\left(\frac{3}{2}a_\delta^3-\frac{1}{2}c_3\right)\right)-\epsilon_2\frac{1}{\sqrt{|D_2|}}N_\delta^{3,6}+\frac{9}{\sqrt{|D|}}c_2.
\end{equation}
If $N_\delta^{3,6}\neq 0$, then $a_\delta^3=c_3$ because $M_3\perp M_6$. In this case \eqref{eqn:sixSixD} gives $N_\delta^{3,6}=10$. This proves that there are $c_6/10=c_3$	elements $\delta$ of order $3$ and norm $2/3\mod 1$ with $N_\delta^{3,6}=10$, while $N_\delta^{3,6}=0$ for all other $\delta$. We let $H$ be the set of isotropic elements in $D_3\cap M_3^\perp$. Let $\gamma\in M_3$ and $\mu\in H$. Then $\gamma+\mu$ obviously has order $3$ and norm $1/3\mod 1$ and $(\gamma+\mu)^\perp\cap M_3=\emptyset$. Applying \eqref{eqn:sixSixC} to $\gamma+\mu$ shows that this is only possibly if $\gamma+\mu\in M_3$. Now let $\gamma_1,\gamma_2\in M_3$. Since these are not orthogonal, exactly one of $\gamma_1+\gamma_2$ and $\gamma_1-\gamma_2$ is isotropic. By replacing $\gamma_2$ with $-\gamma_2$ if necessary, we can assume that $\mu=\gamma_1-\gamma_2$ is isotropic. Pairing $\mu$ with $\eta(\tau)^2\eta(2\tau)^2\eta(3\tau)^2\eta(6\tau)^2$ gives
\begin{align*}
0=&\epsilon_2\frac{1}{\sqrt{|D_3|}}c_1-\epsilon_2\frac{1}{3\sqrt{|D_3|}}\left(\frac{1}{2}a_\mu^3-\frac{1}{2}c_3\right)-\frac{1}{\sqrt{|D|}}c_1+\frac{1}{2\sqrt{|D_2|}}c_2\\
&+\frac{1}{3\sqrt{|D|}}\left(\frac{3}{2}a_\mu^3-\frac{1}{2}c_3\right)-\frac{1}{6\sqrt{|D|}}\left(\frac{3}{2}a_\mu^6-\frac{1}{2}c_6\right),
\end{align*}
which, using that $\mu\in M_6^\perp$, gives $a_\mu^3=c_3$, i.e. $\mu\in M_3^\perp$. This proves that $\mu\in H$ and that $M_3=\pm \gamma+H$ for any $\gamma\in M_3$. Next we show that
\[
M_6=\pm\delta +A_{1/2}+H
\]
where $\delta$ is any element of order $3$ and norm $2/3\mod 1$ with $N_\delta^{3,6}=10$ and $A_{1/2}$ is the set of elements of order $2$ and norm $1/2\mod 1$ in $D$. First note that since $|A_{1/2}|=N(D_2,1)=10$ and $|H|=c_3=2\cdot 3^{n_3/2-1}$ both sides have the same size. It therefore suffices to show that the right hand side is contained in the left. We let $S$ be the set of elements $\delta\in D$ of order $3$ and norm $2/3\mod 1$ with $N_\delta^{3,6}=10$. Since $|A_{1/2}|=N_\delta^{3,6}$, we only have to show that $\pm\delta+H\subset S$. Let $\mu\in H$. We have seen that $M_3=\pm\gamma+H$ for $\gamma\in M_3$, so every element in $H$ is the sum of two elements in $M_3$. Since $M_3\perp M_6$, this shows that $H\subset M_6^\perp$, so in particular $\mu\in M_6^\perp$. Therefore $\mu\delta=0\mod 1$ and $\pm\delta+\mu$ has norm $2/3\mod 1$ for all $\delta\in S$. By definition $\mu$ is also orthogonal to $M_3$, so that $\pm\delta+\mu$ is orthogonal to $M_3$. Then \eqref{eqn:sixSixD} gives $N_{\pm\delta+\mu}^{3,6} = 10$, so $\pm\delta+\mu\in S$, which is what we wanted to show. Finally we prove that $H$ is a group: Let $\mu_1,\mu_2\in H$. Then obviously $\mu_1-\mu_2\in M_3^\perp$, so we only have to show that $\mu_1-\mu_2$ is isotropic. Suppose it is not. Then it has norm $1/3\mod 1$ or $2/3\mod 1$. By replacing $\mu_2$ with $-\mu_2$ if necessary, we can assume that $\mu_1-\mu_2$ has norm $2/3\mod 1$. Then \eqref{eqn:sixSixD} shows that it is in $S$ and therefore not in $M_6^\perp$, which is impossible. Therefore $\mu_1-\mu_2$ is isotropic and hence in $H$, and $H$ is indeed a group. That $F$ must then be a lift of a symmetric strongly reflective form on $D_H=H^\perp/H$ can be seen as in the case with $c_2>0$.

\item If $N=14$, then $L=\II_{4,2}(2_{\II}^{+4}7^{-3})$. Let $\beta\in M_2$. Pairing $F_\beta$ with $h_1(\tau)=\eta(\tau)^5\eta(7\tau)$ gives
\[
0=3-6\frac{1}{\sqrt{|D_2|}}(2a_\beta^2-c_2)-3\frac{\sqrt{7}}{\sqrt{|D_7|}}+6\frac{\sqrt{7}}{\sqrt{|D|}}(2a_\beta^2-c_2),
\]
which in our case gives
$
a_\beta^2=2,
$
i.e. the two elements in $M_2$ are orthogonal. We can also pair $F_\beta$ with $h_2(\tau)=\eta(\tau)^{-2}\eta(2\tau)^7\eta(7\tau)^2\eta(14\tau)$, which gives
\begin{align*}
0=&3-\frac{3}{4\sqrt{|D_2|}}\left(2a_\beta^2-c_2\right)+\frac{21}{4\sqrt{|D_2|}}c_1+\frac{3\sqrt{7}}{\sqrt{|D_7|}}-\frac{3\sqrt{7}}{4\sqrt{|D|}}\left(2a_\beta^{14}-c_{14}\right)\\
&-\frac{3\sqrt{7}}{4\sqrt{|D|}}c_7-\frac{3\sqrt{7}}{4\sqrt{|D|}}\left(2a_\beta^2-c_2\right)+\frac{21\sqrt{7}}{4\sqrt{|D|}}c_1.
\end{align*}
Inserting the known values gives
$
a_\beta^{14}=112=c_{14}$,
i.e. $\beta$ is orthogonal to $M_{14}$. Since this holds for both $\beta\in M_2$, we obtain $M_2\perp M_{14}$. There are exactly $112$ elements of order $14$ and norm $1/14\mod 1$ in $D$ that are orthogonal to $M_2$, namely those in $M_2+A_{4/7}$, where $A_{4/7}$ is the set of elements of order $7$ and norm $4/7\mod 1$ in $D$. Let $H=\{0,\beta_1-\beta_2\}$, where $\beta_1$ and $\beta_2$ are the two elements in $M_2$. Then $M_2=\beta_1+H$ and $M_{14}=\beta_1+H+A_{4/7}$. As in the previous cases it then follows that $F$ is the lift of the unique symmetric strongly reflective modular form $F_H$ with $[(F_H)_0](0)=2$ on $D_H=H^\perp/H$. 
\end{itemize}
\end{proof}
We can now prove Theorem \ref{thm2}.
\begin{proof}[Proof of Theorem \ref{thm2}]
By \ref{thm:converse}, the function $\Psi$ is the image of some strongly reflective modular form $F$ on $D=L'/L$. If $F$ is symmetric, then the result follows from Proposition \ref{prop:classificationSymmetric}. If $F$ is not symmetric, then Proposition \ref{prop:nonSymmetricIsLift}
 shows that $F$ is the lift of a symmetric strongly reflective form $F'$ on some smaller discriminant form $D_H=H^\perp/H$. In all cases $D_H$ can be realised as $N'/N$ for some sublattice $N\subset L$ as described after Theorem \ref{thm:133} and $\Psi$ is therefore the theta lift of $F'$.
\end{proof}

\bibliographystyle{amsalpha}
\bibliography{bibtex}

\providecommand{\bysame}{\leavevmode\hbox to3em{\hrulefill}\thinspace}
\providecommand{\MR}{\relax\ifhmode\unskip\space\fi MR }
\providecommand{\MRhref}[2]{%
  \href{http://www.ams.org/mathscinet-getitem?mr=#1}{#2}
}
\providecommand{\href}[2]{#2}
\begin{thebibliography}{DHS15}

\bibitem[Asa76]{Asai}
Tetsuya Asai, \emph{{On the Fourier coefficients of automorphic forms at
  various cusps and some applications to Rankin's convolution}}, Journal of the
  Mathematical Society of Japan \textbf{28} (1976), no.~1, 48--61.

\bibitem[Bor98]{BorcherdsGrass}
Richard~E. Borcherds, \emph{{Automorphic forms with singularities on
  Grassmannians}}, Inventiones mathematicae \textbf{132} (1998), no.~3,
  491--562.

\bibitem[Bru14]{BruinierConverse}
Jan~H. Bruinier, \emph{{On the converse theorem for Borcherds products}},
  Journal of Algebra \textbf{397} (2014), 315--342.

\bibitem[CS99]{ConwaySloane}
John Conway and Neil J.~A. Sloane, \emph{{Sphere Packings, Lattices and
  Groups}}, third ed., Springer-Verlag New York, 1999.

\bibitem[DHS15]{DHS}
Moritz Dittmann, Heike Hagemeier, and Markus Schwagenscheidt,
  \emph{{Automorphic products of singular weight for simple lattices}},
  Mathematische Zeitschrift \textbf{279} (2015), no.~1, 585--603.

\bibitem[GH14]{GritsenkoUniruled}
Valery~A. Gritsenko and Klaus Hulek, \emph{{Uniruledness of orthogonal modular
  varieties}}, Journal of Algebraic Geometry \textbf{23} (2014), no.~4,
  711--725.

\bibitem[GN02]{GritsenkoKacMoody}
Valery~A. Gritsenko and Viacheslav~V. Nikulin, \emph{{On classification of
  Lorentzian Kac-Moody algebras}}, Russian Mathematical Surveys \textbf{57}
  (2002), no.~5, 921--979.

\bibitem[HBJ94]{Hirzjung}
Friedrich Hirzebruch, Thomas Berger, and Rainer Jung, \emph{{Manifolds and
  Modular Forms}}, Vieweg+Teubner Verlag, 1994.

\bibitem[Ma17]{Ma2}
Shouhei Ma, \emph{{Finiteness of $2$-reflective lattices of signature
  $(2,n)$}}, American Journal of Mathematics \textbf{139} (2017), no.~2,
  513--524.

\bibitem[Ma18]{Ma}
\bysame, \emph{{On the Kodaira dimension of orthogonal modular varieties}},
  Inventiones mathematicae (2018).

\bibitem[Nik80]{Nikulin}
Viacheslav~V. Nikulin, \emph{{Integral symmetric bilinear forms and some of
  their applications}}, Mathematics of the USSR-Izvestiya \textbf{14} (1980),
  no.~1, 103--167.

\bibitem[Rad73]{Rademacher}
Hans Rademacher, \emph{{Topics in Analytic Number Theory}}, first ed.,
  Springer-Verlag Berlin Heidelberg, 1973.

\bibitem[Sch06]{ScheithauerClass}
Nils~R. Scheithauer, \emph{{On the classification of automorphic products and
  generalized Kac-Moody algebras}}, Inventiones mathematicae \textbf{164}
  (2006), no.~3, 641--678.

\bibitem[Sch09]{ScheithauerWeil}
\bysame, \emph{{The Weil representation of $\SL_2(\Z)$ and some of its
  applications}}, International Mathematics Research Notices \textbf{8} (2009),
  1488--1545.

\bibitem[Sch15]{ScheithauerModForms}
\bysame, \emph{{Some constructions of modular forms for the Weil representation
  of $\SL_2(\Z)$}}, Nagoya Mathematical Journal \textbf{220} (2015), 1--43.

\bibitem[Sch17]{ScheithauerSingular}
\bysame, \emph{{Automorphic products of singular weight}}, Compositio
  Mathematica \textbf{153} (2017), no.~9, 1855--1892.

\end{thebibliography}
\end{document}